\documentclass[12pt]{amsart}
\usepackage{amssymb,mathrsfs,diagrams}
\AtEndDocument{\vfill\eject\batchmode}% Suppress font verbiage at end of output
\newtheorem{thm}{Theorem}
\newtheorem*{result}{Theorem}
\newtheorem{cor}{Corollary}[section]
\newtheorem{lemma}{Lemma}[section]
\newtheorem{prop}{Proposition}[section]
\newtheorem{obs}{Observation}[section]
\theoremstyle{definition}
\newtheorem{defn}{Definition}[section]

\newtheorem{remark}{Remark}[section]

\makeatletter
\DeclareSymbolFont{script}{U}{eus}{m}{n}
\DeclareMathSymbol{\Wedge}{0}{script}{"5E}
\DeclareSymbolFont{rmslops}{OT1}{cmr}{m}{sl}
\DeclareSymbolFontAlphabet{\mathrmsl}{rmslops}
\def\operator@font{\mathgroup\symrmslops}
\newenvironment{bulletlist}{\begin{list}{\labelitemi}%
{\setlength{\leftmargin}{\parindent}%
\advance\@listdepth1\relax%
\def\makelabel ##1{\hss \llap {\upshape ##1}}}}{\advance\@listdepth-1\relax%
\end{list}}
\newcommand{\step}[2]{\smallbreak\everypar{\setbox0=\lastbox
\hbox{\textbf{Step #1: #2}.} \everypar{}}}
\makeatother
\newcommand{\acknowledge}{\subsection*{Acknowledgements}}
\newcommand{\calsyma}[2]{\newcommand{#1}{{\mathcal{#2}}}}
\newcommand{\calsymb}[2]{\newcommand{#1}{{\mathscr{#2}}}}
\newcommand{\bbsymb}[2]{\newcommand{#1}{{\mathbb{#2}}}}
\newcommand{\liealg}[2]{\newcommand{#1}{{\mathfrak{#2}}}}
\bbsymb\C{C} \bbsymb\HQ{H}\bbsymb\N{N} \bbsymb\Q{Q}
\bbsymb\R{R} \bbsymb\V{V} \bbsymb\W{W} \bbsymb\Z{Z}
\calsymb\cA{A} \calsymb\cB{B} \calsymb\cC{C} \calsymb\cD{D} \calsymb\cE{E}
\calsymb\cF{F} \calsymb\cG{G} \calsymb\cH{H} \calsymb\cI{I} \calsymb\cJ{J}
\calsymb\cK{K} \calsymb\cL{L} \calsymb\cM{M} \calsymb\cN{N} \calsymb\cO{O}
\calsymb\cP{P} \calsyma\cQ{Q} \calsymb\cR{R} \calsymb\cS{S} \calsymb\cT{T}
\calsymb\cU{U} \calsymb\cV{V} \calsymb\cW{W} \calsymb\cX{X} \calsymb\cY{Y}
\calsymb\cZ{Z}
\liealg\gl{gl}\liealg\sgl{sl}\liealg\symp{sp}
\liealg\g{g}\liealg\p{p}
\newcommand{\eps}{\varepsilon}
\renewcommand{\geq}{\geqslant} \renewcommand{\leq}{\leqslant}
\renewcommand{\emptyset}{\varnothing}
\newcommand{\into}{\hookrightarrow}

\DeclareMathOperator{\ri}{r}
\DeclareMathOperator{\Ri}{R}
\DeclareMathOperator{\Wi}{W}
\DeclareMathOperator{\SO}{SO}
\DeclareMathOperator{\CO}{CO}

\DeclareMathOperator{\PGL}{PGL}
\DeclareMathOperator{\Aff}{Aff}

\DeclareMathOperator{\im}{im}
\DeclareMathOperator{\spanu}{span}
\DeclareMathOperator{\sym}{sym}
\DeclareMathOperator{\tr}{tr}
\DeclareMathOperator{\id}{id}
\DeclareMathOperator{\Hom}{Hom}
\newcommand{\cx}{^{\scriptscriptstyle\C}}
\newcommand{\ur}{_{\;\,\R}}
\newcommand{\Sc}{S\cx}
\renewcommand{\d}{{\mathrmsl d}}
\newcommand{\Proj}{{\mathrmsl P}}
\newcommand{\Gr}{{\mathrmsl{Gr}}}

\newcommand{\hp}[1][^]{#1{1,0}}
\newcommand{\am}[1][^]{#1{0,1}}
\newcommand{\Vhp}{\cV\hp}
\newcommand{\Vam}{\cV\am}
\newcommand{\mfd}{M}
\newcommand{\php}{\phi\hp[_]}
\newcommand{\phm}{\phi\am[_]}
\newcommand{\pip}{\pi\hp[_]}
\newcommand{\pim}{\pi\am[_]}
\newcommand{\Lp}{\cL\hp[_]}
\newcommand{\Lm}{\cL\am[_]}
\newcommand{\zers}{\underline{0}}
\newcommand{\infs}{\underline{\infty}}
\newcommand{\Lam}{{\mathit\Lambda}}

\newcommand{\Lmp}{\Lam^+}
\newcommand{\Lmm}{\Lam^-}
\newcommand{\sub}{\subseteq}
\newcommand{\tp}{\otimes}
\newcommand{\ds}{\oplus}
\newcommand{\Nb}{\cN}
\newcommand{\qD}{{\mathfrak D}}
\newcommand{\ul}{{\mathfrak U}}
\newcommand{\qs}{\cQ}
\newcommand{\rs}{\rho}
\newcommand{\tx}{\tilde x}
\newcommand{\tz}{\tilde z}
\newcommand{\ra}[1]{{\raise6pt\hbox{$#1$}}}
\newcommand{\swb}{\cU}

\newcommand{\iI}{\boldsymbol{i}}
\newcommand{\vecD}[3][D]{\bigl[\begin{smallmatrix} #2\\
#3 \end{smallmatrix}\bigr]_{#1}}
\newcommand{\VecD}[3][D]{\Bigl[\begin{matrix} #2\\
#3 \end{matrix}\Bigr]_{#1}}
\newcommand{\abrack}[1]{[\mkern-3mu[#1]\mkern-3mu]}
\newcommand{\mult}{^{\scriptscriptstyle\smash[b]{\times}}}
\begin{document}
\title[Projective geometry and the quaternionic Feix--Kaledin construction]
{Projective geometry\\ and the quaternionic Feix--Kaledin construction}
\author[A.W. Bor\'owka]{Aleksandra W. Bor\'owka}
\address{Aleksandra W. Bor\'owka \\
Institute of Mathematics\\ Jagiellonian University\\
ul. prof. Stanislawa Lojasiewicza 6\\
30-348 Krak\'ow\\ Poland}
\email{Aleksandra.Borowka@uj.edu.pl}
\author[D.M.J. Calderbank]{David M. J. Calderbank}
\address{David M. J. Calderbank \\ Department of Mathematical Sciences\\
University of Bath\\ Bath BA2 7AY\\ UK}
\email{D.M.J.Calderbank@bath.ac.uk}
%\excludecomment{abstract}\def\endabstract{}% Comment out to include abstract
\begin{abstract}
Starting from a complex manifold $S$ with a real-analytic c-projective
structure whose curvature has type $(1,1)$, and a complex line bundle $\cL\to
S$ with a connection whose curvature has type $(1,1)$, we construct the
twistor space $Z$ of a quaternionic manifold $M$ with a quaternionic circle
action which contains $S$ as a totally complex submanifold fixed by the
action.  This extends a construction of hypercomplex manifolds, including
hyperk\"ahler metrics on cotangent bundles, obtained independently by
B.~Feix~\cite{Feix,Feix2,Feix3} and D.~Kaledin~\cite{Kal,Kal2}.

When $S$ is a Riemann surface, $M$ is a self-dual conformal $4$-manifold, and
the quotient of $M$ by the circle action is an Einstein--Weyl manifold with an
asymptotically hyperbolic end~\cite{JT,Le}, and our construction coincides
with the construction presented by the first author in~\cite{Bor}. The
extension also applies to quaternionic K\"ahler manifolds with circle actions,
as studied by A.~Haydys~\cite{Hay} and N.~Hitchin~\cite{Hit2}.
\end{abstract}
\maketitle
%\tableofcontents

\section*{Introduction}

The construction of hyperk\"ahler metrics on cotangent bundles of K\"ahler
manifolds has a distinguished history, going back to E.~Calabi's metric on the
cotangent bundle of $\C\Proj^n$~\cite{Calabi}, and its generalizations to
complex semisimple Lie groups and their flag
varieties~\cite{Biq,Kron,Kron2,Nak}.  General constructions were provided
independently by B.~Feix~\cite{Feix,Feix3} and D.~Kaledin~\cite{Kal,Kal2}, who
showed that on a complex manifold $S$, any real-analytic K\"ahler metric
induces a hyperk\"ahler metric on a neighbourhood of the zero section in
$T^*S$. In fact, both authors further established (see~\cite{Feix2} in Feix's
case) that any real-analytic complex affine connection on $S$ with curvature
of type $(1,1)$ induces a hypercomplex structure on a neighbourhood of the
zero section in $TS$, invariant under the action of the circle $S^1$.

An important generalization of hypercomplex manifolds are quaternionic
manifolds~\cite{Sal}, which are of particularly great interest when they admit
quaternionic K\"ahler metrics.  While the most famous problem in the area is
the LeBrun--Salamon conjecture~\cite{LeSa} on the classification in the
compact positive scalar curvature case, recently much attention has been given
to correspondences between $S^1$-invariant quaternionic K\"ahler and
hyperk\"ahler metrics in connection with theoretical physics, e.g., string
theory duality~\cite{ACDM,CFG,Hay,Hit2,MaSw}. Hence it is desirable to
generalize the Feix--Kaledin results to the quaternionic context.

In this paper we provide such a generalization, using Feix's twistorial
method~\cite{Feix,Feix2} (cf.~also~\cite{LeBr}) and two new ingredients of
independent interest. First, the complex manifold $S$ is endowed with a weaker
and more subtle geometric structure than an affine connection, namely a
\emph{c-projective structure} (see~\cite{CEMN}). Secondly, we introduce into
the construction a \emph{twist} by a holomorphic line bundle $\cL\to S$ with
connection. Then, for a given c-projective manifold $S$ we obtain a family of
quaternionic manifolds with $S^1$ symmetry containing $S$ as a totally complex
submanifold. Moreover, we prove that any such manifold arises in this way on a
neighbourhood of a generic fixed totally complex submanifold. In summary, we
develop a natural projective-geometric framework for the Feix--Kaledin results
which encompasses the recently studied $S^1$-invariant quaternionic K\"ahler
manifolds~\cite{Hay,Hit2} and describes their behaviour near a fixed
submanifold.

We present the general construction (Theorem~\ref{maintheorem},
cf.~\cite{BorPhD}) and its converse characterization
(Theorem~\ref{conversetheorem}) in Section~\ref{s:moc}, where we also motivate
the projective-geometric framework.  We begin by comparing the
``hypercomplexification'' of $S$ in $TS$ (cf.~\cite{Biel}) to the
complexification of a real-analytic manifold. In particular, results of
R.~Bielawski~\cite{Biel} and R.~Sz\"oke~\cite{Szoke} imply that a
real-analytic projective manifold $M$ has a complexification $M\cx\sub TM$
which meets the tangent bundle to any geodesic in a holomorphic submanifold
(Theorem~\ref{thm:pcomplex}), illustrating the role of projective geometry
already in this setting. We then show (Theorem~\ref{thm:totallycomplex}) that
the natural structure induced on a maximal totally complex submanifold of a
quaternionic manifold is a c-projective structure, and explain the general
construction with reference to the model example of quaternionic projective
space $\HQ\Proj^n$, which has $\C\Proj^n$ as a maximal totally complex
submanifold. The role of the twist is already apparent here, as the
Feix--Kaledin metric associated to $\C\Proj^n$ is the Calabi metric on
$T^*\C\Proj^n$, not $\HQ\Proj^n$.

The remainder of the paper gives details, applications and examples.  As the
construction uses diverse ingredients, for the convenience of the reader we
provide essential background on projective geometry in Section~\ref{s:bg} and
on quaternionic twistor theory in Section~\ref{s:qtt}.  We give the remaining
details of the proof of Theorem~\ref{maintheorem}, and prove
Theorem~\ref{conversetheorem}, in Section~\ref{s:detail}.
Section~\ref{s:examap} illustrates the scope of our construction through
examples and connections with other results in the area. Here we discuss first
the complex grassmannian, which is yet another twist starting from
$\C\Proj^n$, and not even locally hyperk\"ahler. Then we explain how Feix's
construction arises as a special case and also how the $4$-dimensional case is
connected with LeBrun's asymptotically hyperbolic Einstein--Weyl
structures~\cite{Bor,Le}. To conclude, in Theorem~\ref{Swann}, we relate the
quaternionic and hypercomplex constructions via twisted Swann
bundles~\cite{Swann,Joyce,Joyce2,PPS,Sal} and twisted Armstrong
cones~\cite{Arm}, and then use this in Theorem~\ref{thm:qK} to characterize
quaternionic K\"ahler metrics arising in the Haydys--Hitchin
correspondence~\cite{Hay,Hit2}.

\section{Motivation and overview of the construction}\label{s:moc}

\subsection{Complexification and projective geometry}\label{s:cpg}
 
Any real-analytic $n$-manifold $\mfd$ has a complexification, which is a
holomorphic $n$-manifold $\mfd\cx$ containing $\mfd$ as the fixed point set
real structure $\rs\colon\mfd\cx\to \mfd\cx$ (an antiholomorphic map with
$\rs^2=\id$). The underlying complex manifold of $\mfd\cx$, which is a real
$2n$-manifold $\mfd\cx\ur$ with an integrable complex structure $J$, has
$\mfd$ as a \emph{totally real submanifold}, i.e., $T\mfd\cap J(T\mfd)=0$, so
$T \mfd\cx\ur|_{\mfd}=T\mfd\ds J(T\mfd)$.  Since $J(T\mfd)\cong T\mfd$ is the
normal bundle to $\mfd$ in $\mfd\cx\ur$, there is a local isomorphism along
$\mfd$ between $\mfd\cx\ur$ and $T\mfd$, where $\mfd$ is identified with the
zero section in $T\mfd$, along which $J$ is an isomorphism between horizontal
and vertical tangent spaces in $T(T\mfd)$.  Such a complexification of $\mfd$
inside $T\mfd$ is unique up to unique local automorphism inducing the identity
to first order along $\mfd$; furthermore, the complexification can be
determined uniquely by choosing an affine connection $D$ on $\mfd$ and
requiring that the tangent map of any geodesic is
holomorphic~\cite{Biel,Szoke}. However, the \emph{unparametrized} geodesics of
$D$ depend only on its projective class in the following sense.

\begin{defn}\label{d:proj} A \emph{projective manifold} is a manifold $\mfd$
with \emph{projective structure}, i.e., a \emph{projective equivalence class}
$\Pi_r=[D]_r$ of torsion-free affine connections, where $\tilde D\sim_r D$ if
there is a $1$-form $\gamma\in\Omega^1(\mfd)$ such that for all vector fields
$X,Y\in\Gamma(T\mfd)$,
\begin{equation}\label{eq:proj}
\tilde D_XY= D_XY +\abrack{X,\gamma}^r(Y),\quad\text{where}\quad
\abrack{X,\gamma}^r(Y)=\gamma(X)Y+\gamma(Y)X.
\end{equation}
\end{defn}
Hence the results of Bielawski and Sz\"oke~\cite{Biel,Szoke} have the
following consequence.
\begin{thm}\label{thm:pcomplex} A real-analytic projective manifold $\mfd$ has
a complexification $\mfd\cx\sub TM$ which meets the tangent bundle to any
geodesic in $M$ in a holomorphic submanifold.
\end{thm}

\subsection{Quaternionic manifolds and totally complex submanifolds}\label{s:qm-tcs}

Recall~\cite{Sal} that a \emph{quaternionic structure} on a $4n$-manifold $M$
is a bundle $\qs$ of Lie subalgebras of the endomorphism bundle $\gl(TM)$ of
$TM$ which is pointwise isomorphic to the Lie algebra $\symp(1)$ of imaginary
quaternions acting on $\R^{4n}\cong\HQ^n$; a \emph{quaternionic connection}
$\qD$ on $(M,\qs)$ is a torsion-free affine connection preserving $\qs$.  If
$(M,\qs)$ admits a quaternionic connection (satisfying a curvature condition
when $n=1$ which we discuss later), we say it is a \emph{quaternionic
  manifold}.

A submanifold $S$ of $(M,\qs)$ is \emph{totally complex}~\cite{Alex2} if there
is a section $J$ of $\qs|_S$ with $J^2=-\id$ such that:
\begin{bulletlist}
\item $J(TS)\sub TS$ (so that $J$ is an almost complex structure on $S$);
\item for all $I\in J^\perp$, $I(TS)\cap TS = 0$, where $J^\perp:=\{I\in \qs :
IJ = -JI\}$.
\end{bulletlist}

If $M$ has real dimension $4n$, it follows that $S$ has real dimension $\leq
2n$. If $S$ is \emph{maximal}, i.e., dimension $2n$, then $TM|_S=TS\ds NS$
where $(NS)_u=I(T_uS)$ for any nonzero $I\in J^\perp_u$. (Any other element of
$J^\perp_u$ is a pointwise linear combination of $I$ and $IJ$, so $(NS)_u$ is
independent of the choice of $I$, and the map $J^\perp_u\times T_u S\to
(NS)_u; (I,X)\mapsto IX$ induces an isomorphism $J^\perp_u\tp_\C T_uS\cong
(NS)_u$, where $J^\perp_u$ and $T_u S$ are complex vector spaces via right
multiplication by $J$ and its left action respectively.)

\begin{lemma}\label{lem:totallycomplex} Let $S$ be a maximal totally complex
submanifold of $(M,\qs)$ and $\qD$ a quaternionic connection, and let $\pi
\colon TM|_S\to TS$ be the projection along $NS$. Then the projection $D_X
Y:=\pi(\qD_XY)$, for vector fields $X,Y$ on $S$, defines a torsion-free
complex connection \textup(i.e., $DJ=0$\textup) on $S$, and hence $J$ is
integrable on $S$.
\end{lemma}
\begin{proof} Clearly $D$ is a torsion-free connection on $S$: for any vector
fields $X,Y$ on $S$, $D_XY-D_YX = \pi([X,Y])=0$. Furthermore,
\begin{align*}
(D_XJ)Y=D_X(JY)-JD_XY&=\pi(\qD_X(JY))-J\pi(\qD_XY)\\
&=\pi (\qD_XJ)Y+(\pi J-J\pi)\qD_XY=0,
\end{align*}
since $\qD_XJ$ is a section of $J^\perp$, and $J$ commutes with $\pi$.
\end{proof}
If $\tilde\qD$ is another quaternionic connection on $M$, it is well
known~\cite{Alex} that there is a $1$-form $\gamma$ on $M$ such that
$\tilde\qD_X Y = \qD_X Y + \abrack{X,\gamma}^q(Y)$, where
\begin{equation}\label{eq:quat}
\textstyle \abrack{X,\gamma}^q(Y):=\frac{1}{2}(\gamma(X)Y+\gamma(Y)X
-\sum_{i=1}^3\bigl(\gamma(J_iX)J_iY+\gamma(J_iY)J_iX)\bigr)
\end{equation}
where $J_1,J_2,J_3$ is any local anticommuting frame of $\qs$ with
$J_i^{\,2}=-\id$. Thus, given one quaternionic connection $D$, we can
construct all others using $\abrack{\cdot,\cdot}^q$.

For a maximal totally complex submanifold $S\sub M$, we may take the
anticommuting frame defined by the given complex structure $J$ preserving
$TS$, a local section $I$ of $J^\perp$ with $I^2=-\id$, and $K=IJ$. Then for
vector fields $X,Y$ along $S$, we compute
\begin{align} \nonumber
\pi(\tilde\qD_X Y -\qD_X Y)=&\pi(\abrack{X,\gamma}^q(Y))
=\abrack{X,\gamma}^c(Y),\qquad\text{where}\\
\abrack{X,\gamma}^c(Y)
:=&\tfrac12 \bigl(\gamma(Y)Z+\gamma(Z)Y-(\gamma(JY)JZ+\gamma(JZ)JY)\bigr)
\label{eq:cproj}\end{align}
and we use $\pi(IX)=\pi(KX)=0$. This prompts the following definition.

\begin{defn} A \emph{c-projective manifold} is a manifold $S$ with an
integrable complex structure $J$ and a \emph{c-projective structure}, i.e., a
\emph{c-projective equivalence} class $\Pi_c=[D]_c$ of torsion-free complex
connections, where $\tilde D\sim_c D$ if there is a $1$-form $\gamma$ such
that for all vector fields $X,Y$ on $S$, $\tilde D_X Y = D_X Y
+\abrack{X,\gamma}^c(Y)$.
\end{defn}
This is complex, though not necessarily holomorphic, analogue of a real
projective structure (see \S\ref{s:parabolic} and~\cite{CEMN,Hrdina,Ish,Y},
some of which use misleading terms ``holomorphically projective'' and
``h-projective''). The observations above imply the following.

\begin{thm}\label{thm:totallycomplex} Let $S$ be a maximal totally complex
submanifold of a quaternionic manifold $(M,\qs)$. Then $S$ is a c-projective
manifold, whose c-projective structure consists of the connections induced by
quaternionic connections on $M$ via Lemma~\textup{\ref{lem:totallycomplex}}.
\end{thm}
Since the normal bundle of $S$ in $M$ is isomorphic to $TS\tp_\C J^\perp$, a
neighbourhood of $S$ in $M$ is isomorphic to a neighbourhood of the zero
section in $TS\tp_\C J^\perp$.

We show in~\S\ref{s:parabolic} that the \emph{c-projective curvature} of $S$
has type $(1,1)$ with respect to $J$.  Conversely, as we shall see, the
quaternionic Feix--Kaledin construction exhibits every real-analytic
c-projective manifold with type $(1,1)$ c-projective curvature as a maximal
totally complex submanifold of a quaternionic manifold.

\subsection{The model example and the twistor construction}
\label{motex}

Given a quaternionic vector space $W\cong \HQ^{n+1}$, its quaternionic
projectivization $M=\Proj_\HQ(W)\cong \HQ\Proj^n$ has a canonical quaternionic
structure: a point $H\in M$ is a $1$-dimensional quaternionic subspace of $W$,
and its tangent space $T_H M$ is the space of quaternionic linear maps $H\to
W/H$, which is itself a quaternionic vector space; the action of the imaginary
quaternions on $T_H M$ defines an $\symp(1)$ subalgebra
$\qs_H\cong\sgl(H,\HQ)\sub\gl(T_H M)$. Now let $W_\C$ be the underlying
complex vector space of $W$ with respect to one of its complex structures
$J$. Then there is a natural map $\pi_M$ from $Z=\Proj(W_\C)\cong
\C\Proj^{2n+1}$ to $M$ whose fibre at $H\in M$ is $\Proj(H_\C)\cong
\C\Proj^1$, which is isomorphic to the $2$-sphere of unit imaginary
quaternions in $\sgl(H,\HQ)$. These fibres are fixed by the antiholomorphic
involution of $Z$ induced by any nonzero element of $J^\perp$.

Now let $W_\C=W\hp\ds W\am$, where $W\hp\cong W\am\cong \C^{n+1}$ are maximal
totally complex subspaces of $W$ with respect to the chosen complex structure
$J$, i.e., $JW\hp = W\hp$, $JW\am= W\am$, and $I W\hp=W\am$ for any nonzero
$I\in J^\perp$.  Then $\Proj(W\hp)$ and $\Proj(W\am)$ are disjoint projective
$n$-subspaces of $Z=\Proj(W_\C)$, and $S:=\pi_M(\Proj(W\hp))=
\pi_M(\Proj(W\am))\cong \C\Proj^n$ is a maximal totally complex submanifold of
$M\cong \HQ\Proj^n$.

\begin{prop} $Z\setminus\Proj(W\hp)$ is canonically isomorphic to \textup(the
total space of\textup) the vector bundle $\Hom(\cO_{W\am}(-1),W\hp)\to
\Proj(W\am)$, with fibre $\Hom(\tx,W\hp)$ over $\tx\in\Proj(W\am)$, and
similarly $Z\setminus\Proj(W\am)\cong
\Hom(\cO_{W\hp}(-1),W\am)\to\Proj(W\hp)$. Furthermore the blow-up of $Z$ along
$\Proj(W\hp)\sqcup\Proj(W\am)$ is canonically isomorphic to the
$\C\Proj^1$-bundle
\begin{equation*}
\smash{\hat Z}:=\Proj(\cO(-1,0)\ds\cO(0,-1))\to\Proj(W\hp)\times\Proj(W\am),
\end{equation*}
whose fibre over $(x,\tx)$ is $\Proj(x\ds \tx)$.
\end{prop}
\begin{proof} The fibre of the map $Z\setminus\Proj(W\hp)\to \Proj(W\am);
[w+\tilde w]\mapsto [\tilde w]$ over $\tx\in \Proj(W\am)$ is
$\Proj(W\hp\ds\tx)\setminus\Proj(W\hp)$. Any $1$-dimensional subspace of
$W\hp\ds\tx$ transverse to $W\hp$ is the graph of linear map $\tx\to W\hp$,
yielding an isomorphism $\Proj(W\hp\ds\tx)\setminus\Proj(W\hp)\to
\Hom(\tx,W\hp)$.  The isomorphism of $Z\setminus\Proj(W\am)$ with
$\Hom(\cO_{W\hp}(-1),W\am)$ is analogous, and $\smash{\hat Z}$ is the blow-up
of $Z$ because (see~\S\ref{s:pb-blow-up}) the blow-up of a vector space $E$ at
the origin is isomorphic to the total space of the tautological bundle
$\cO_E(-1)\to\Proj(E)$.
\end{proof}
Thus $Z$ may be obtained from $\Proj(W\hp)\times\Proj(W\am)$ by gluing
together the vector bundles $\Hom(\cO_{W\hp}(-1),W\am)\to\Proj(W\hp)$ and
$\Hom(\cO_{W\am}(-1),W\hp)\to\Proj(W\am)$ to obtain a blow-down of
$\Proj(\cO(-1,0)\ds\cO(0,-1))$ along its two canonical (``zero and infinity'')
sections.  Each fibre $\Proj(x\ds \tx)$ then maps to a projective line in $Z$
with normal bundle isomorphic to $\cO_{x\ds \tx}(1) \tp\C^{2n}$, and these are
the fibres of $Z$ over $S\sub M$.

This picture generalizes using an extension to quaternionic manifolds,
introduced by S.~Salamon~\cite{Sal0,Sal}, of Penrose's twistor theory for
self-dual conformal manifolds~\cite{AHS,Pen}.  The \emph{twistor space} of a
quaternionic $4n$-manifold $(M,\qs)$---or, for $n=1$, a self-dual conformal
manifold---is the total space $Z$ of the $2$-sphere bundle $\pi_M\colon Z\to
M$ of elements of $\qs$ which square to $-1$. Salamon showed that $Z$ admits an
integrable complex structure (and hence is a holomorphic $(2n+1)$-manifold)
such that the involution $\rs$ of $Z$ sending $J$ to $-J$ is antiholomorphic,
and the fibres of $\pi_M$ are \emph{real twistor lines}, i.e., holomorphically
embedded, $\rs$-invariant projective lines with normal bundle isomorphic to
$\cO(1)\tp\C^{2n}$. The following converse will be crucial to our main
construction.

\begin{result} Let $Z$ be a holomorphic $(2n+1)$-manifold equipped with an
antiholomorphic involution $\rs\colon Z\to Z$ containing a real twistor line
$u$ on which $\rs$ has no fixed points. Then the space of such real twistor lines
\textup(i.e., those with no fixed points of $\rho$\textup) is a
$4n$-dimensional quaternionic manifold $(M,\qs)$ such that $(Z,\rs)$ is
locally isomorphic to its twistor space.
\end{result}

For hyperk\"ahler and quaternionic K\"ahler manifolds, this result is due to
N.~Hitchin et al.~\cite{HKLR} and C.~LeBrun~\cite{LeBrun} respectively.
H. Pedersen and Y-S. Poon~\cite{PP} establish an extension to general
quaternionic manifolds, although they assume that $Z$ is foliated by real
twistor lines.  However, the Kodaira deformation space~\cite{KK} of $u$ is a
holomorphic $4n$-manifold $M\cx$ with a real structure $\rs_M$ whose fixed
points are real twistor lines. It follows that the real twistor lines form a
real-analytic submanifold $M$ of $M\cx$ with real dimension $4n$, which is
enough to establish the above result, following~\cite{BE,HKLR,LeBrun,PP}.

\subsection{The quaternionic Feix--Kaledin construction}\label{s:qfkc}

Let $S$ be a $2n$-manifold equipped with an integrable complex structure $J$
and a real-analytic c-projective structure $\Pi_c$. Our goal is to build the
twistor space $Z$ of a quaternionic manifold $M$ from a projective line bundle
$\hat Z= \Proj(\Lp^*\ds\Lm^*) \xrightarrow{p} \Sc$, where $\Sc$ is a
complexification of $S$. The fibres of $p$ over $\Sc$ are projective lines in
$\hat Z$ with trivial normal bundle $\cO\tp\C^{2n}$, but if we map them into a
suitable blow-down $Z$ of $\hat Z$, along ``zero'' and ``infinity'' sections
$\zers=\Proj(\Lp^*\ds 0)$ and $\infs=\Proj(0\ds\Lm^*)$, then their images in
$Z$ will have normal bundle $\cO(1)\tp\C^{2n}$.

In the model example, $\Sc$ is a product of projective spaces, and $\Lp$ and
$\Lm$ are dual to tautological line bundles over the factors. In general, it
will be an open subset of a projective bundle in two different ways, and the
line bundles $\Lp$ and $\Lm$ will be dual to fibrewise tautological line
bundles over these projective bundles. There is some freedom in the choice of
$\Lp$ and $\Lm$, which we parametrize by an auxiliary complex line bundle
$\cL\to S$ equipped with a real-analytic complex connection $\nabla$. We
proceed in several steps.

\step1{Complexification} First we introduce a complexification of $S$, i.e., a
holomorphic manifold $\Sc$ with $S$ as the fixed point set of an
antiholomorphic involution---see~\S\ref{s:complexification}. Since $S$ is a
complex manifold, it has an essentially canonical complexification by
embedding it as the diagonal in $S\hp\times S\am$, where $S\hp$ denotes $S$
with the holomorphic structure induced by $J$ and $S\am=\overline{S\hp}$ is
its conjugate (with the holomorphic structure induced by $-J$) so that
transposition is an antiholomorphic involution of $S\hp\times S\am$.  However,
the c-projective structure $\Pi_c$ on $S$ and connection $\nabla$ on $\cL$ may
only extend to a tubular neighbourhood of the diagonal in $S\hp\times S\am$,
so we let $\Sc$ be such a neighbourhood, with extensions $\Pi_c\cx$ and
$\nabla\cx$ of $\Pi_c$ and $\nabla$. Thus $\Sc$ has transverse $(0,1)$ and
$(1,0)$ foliations, which are the fibres of the projections $\pip\colon \Sc\to
S\hp$ and $\pim\colon \Sc\to S\am$. We let $\Lp$ and $\Lm$ be the pullbacks to
$\Sc$ of $\cL\tp\cO_S(1)\to S=S\hp$ and its conjugate over $S\am$, where
$\cO_S(1)^{\tp(n+1)}=\Wedge^n TS$. (In examples, it can happen that $\cL$ and
$\cO_S(1)$ are not globally defined on $S$, but their tensor product is.)

As explained in Proposition~\ref{p:indproj}, the algebraic bracket
$\abrack{\cdot,\cdot}^c$ restricts to $\abrack{\cdot,\cdot}^r$ on the leaves
of the $(0,1)$ and $(1,0)$ foliations and so restrictions of connections in
$\Pi_c\cx$ induce projective structures, and hence projective Cartan
connections $\cD$, along these
leaves---see~\S\ref{s:parabolic}--\S\ref{s:pc-affine}.  In fact, as explained
in \S\ref{s:cproj}, we couple these connections to the connection $\nabla\cx$
on $\cL\cx$ to obtain connections $\cD^\nabla$ on the bundles of $1$-jets of
$\Lm$ and $\Lp$ along the leaves of the $(0,1)$ and $(1,0)$ foliations
respectively.

\step2{Development} We now introduce the fundamental assumption that $\Pi_c$
and $\nabla$ have (curvature of) type $(1,1)$ with respect to the complex
structure $J$ on $S$---see \S\ref{s:cproj}.  By
Proposition~\ref{p:flatindproj}, the coupled projective Cartan connections
$\cD^\nabla$ are flat along the leaves of the $(0,1)$ and $(1,0)$ foliations.
Since these leaves are assumed to be contractible, hence simply connected, the
rank $n+1$ bundles $J^1\Lm$ and $J^1\Lp$ are trivialized by parallel sections
along the $(0,1)$ and $(1,0)$ foliations respectively.

\begin{defn}\label{d:affine} The bundle $\Aff(\Lm)\to S\hp$ of \emph{affine
sections} along the leaves of the $(0,1)$ foliation (the fibres of $\pip$)
is the bundle whose fibre at $x\in S\am$ is the space of sections $\ell$
of $\Lm$ over $\pip^{-1}(x)$ such that $j^1\ell$ is $\cD^\nabla$-parallel. The
bundle $\Aff(\Lp)\to S\am$ is defined similarly.  We further define
$\Vam:=\Aff(\Lm)^*\tp\Lp\to S\hp$ and $\Vhp:=\Aff(\Lp)^*\tp\Lm\to S\am$.
\end{defn}

The \emph{evaluation maps} $\pip^*\Aff(\Lm)\to \Lm$ and
$\pim^*\Aff(\Lp)\to\Lp$ over $\Sc$ send an affine section along a leaf to its
value at a point on that leaf. Dual to these are line subbundles
$\Lm^*\into\pip^*\Aff(\Lm)^*$ and $\Lp^*\into\pim^*\Aff(\Lp)^*$ over $\Sc$,
and hence fibrewise \emph{developing maps} from $\Sc$ to $\Proj(\Vam)$ over
$S\hp$, or from $\Sc$ to $\Proj(\Vhp)$ over $S\am$, sending a point of $\Sc$
to the fibre of $\Lm^*\tp\Lp$ in $\Vam$, or $\Lp^*\tp\Lm$ in $\Vhp$
respectively. The developing maps are local diffeomorphisms, so we may assume
(shrinking $\Sc$ if necessary) that they embed $\Sc$ as open subsets of
$\Proj(\Vam)$ and $\Proj(\Vhp)$ respectively. These induce embeddings of the
line bundles $\Lm^*\tp\Lp$ and $\Lp^*\tp\Lm$ into the tautological line
bundles $\cO_{\Vam}(-1)\to\Proj(\Vam)$ and $\cO_{\Vhp}(-1)\to\Proj(\Vhp)$.

\step3{Blow-down} To blow $\hat Z$ down along $\zers$ and $\infs$, we make
following definition.
\begin{defn}\label{phi} Let $\phm\colon\hat Z\setminus\infs\to\Vam$
and $\php\colon\hat Z\setminus\zers\to\Vhp$ be the restrictions, to $\hat
Z\setminus\infs\cong\Lm^*\tp\Lp$ and $\hat Z\setminus\zers\cong\Lp^*\tp\Lm$
respectively, of the blow-downs $\cO_{\Vam}(-1)\to\Vam$ and
$\cO_{\Vhp}(-1)\to\Vhp$ of zero sections of tautological line bundles.
\end{defn}
On the complement of $\zers\sqcup\infs$, the blow-down maps $\phm$ and $\php$
are biholomorphisms onto their image---see~\S\ref{s:pb-blow-up}. However,
since $\Sc$ typically embeds as a proper open subset of $\Proj(\Vam)$ and
$\Proj(\Vhp)$, the images of $\phm$ and $\php$ are cones in each fibre of
$\Vam$ and $\Vhp$ (see Remark~\ref{rem:blowup}), hence singular along the zero
sections. As a first attempt to fix this problem, we could replace these
images by $\Vam$ and $\Vhp$ themselves, and then glue these two vector bundles
together by identifying $\phm(z)$ with $\php(z)$ for $z\in\hat
Z\setminus(\zers\sqcup\infs)$.  Unfortunately the space obtained in this way
is typically not Hausdorff.  We repair this by gluing instead open subsets
$Z\am\sub\Vam$ and $Z\hp\sub\Vhp$ as follows.

\begin{defn}\label{d:Z} Let $U\am$ and $U\hp$ be tubular neighbourhoods of the
zero section in $\Vam$ and $\Vhp$ respectively, such that
\begin{equation}\label{eq:Zcond}
\phm^{-1}(U\am)\cap \php^{-1}(U\hp) = \emptyset
\end{equation}
and define
\begin{gather*}
Z\am=\im\phm\cup U\am,\qquad Z\hp=\im\php\cup U\hp, \qquad Z=
Z\am\mathop{\sqcup}_{\sim} Z\hp,
\end{gather*}
where $\phm(z)\sim \php(z)$ for all $z\in \hat Z\setminus(\zers\sqcup\infs)$.
This gluing induces a map
\begin{equation}
\phi\colon \hat Z=\Proj(\Lp^*\ds\Lm^*)\to Z,
\end{equation}
whose restriction to any leaf of the $(0,1)$ foliation is an isomorphism away
from $\zers$, and whose restriction to any leaf of the $(1,0)$ foliation is an
isomorphism away from $\infs$.
\end{defn}

\begin{remark} \label{Propofphi} Via the developing maps, $\phm$ and $\php$
are restrictions of the blow-down maps which contract $2n$-dimensional zero
sections of $\Lm^*\tp\Lp$ and $\Lp^*\tp\Lm$ to $n$-dimensional zero sections
of $\Vam$ and $\Vhp$. The multiplicative parts $(\Lm^*\tp\Lp)\mult$ and
$(\Lp^*\tp\Lm)\mult$ are both isomorphic to $\hat Z\setminus
(\zers\sqcup\infs)$, the composite of these isomorphisms being the inversion
map $\ell\mapsto 1/\ell$. Fibrewise, $Z\am$ and $Z\hp$ look like cones with
small balls added around the origin, and they are glued along the cones by
inversion. 
\end{remark}
The following diagram summarizes the construction of $Z$, where the hooked
arrows are open embeddings, and the other arrows are fibrations or blow-downs.
The left-right symmetry in the diagram corresponds to interchanging the
$(1,0)$ and $(0,1)$ directions.

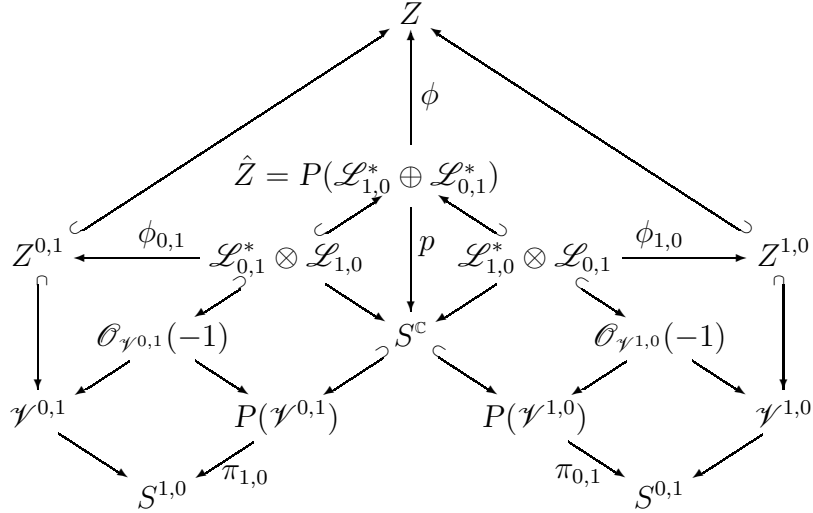
\begin{diagram}[height=1.3em,width=2em,nohug,tight]
&  &        &    &          &   & Z \\
&  &        &    &  &\ruInto(6,6)&   &\luInto(6,6)\\
&  &        &    &          &  &\uTo_\phi \\
\\
&  &        &    &          &  &\llap{$\hat Z=\Proj(\Lp^*$}\ds\rlap{$\Lm^*)$}\\
&  &         &    &       &\ruInto&     &\luInto\\
Z\am& &\lTo^{\phm}& &\Lm^*\tp\Lp& &\dTo_{\ra p}& &\Lp^*\tp\Lm& &\rTo^{\php}& &Z\hp\\
\dIntoA&  &     &\ldInto&   &\rdTo&   &\ldTo&  &\rdInto& & &\dInto\\ 
&  &\cO_{\Vam}(-1)& &          &    &\Sc&     &  &   &\cO_{\Vhp}(-1) & \\
&\ldTo&       &\rdTo&       &\ldInto&&\rdInto& &\ldTo&       &\rdTo& \\
\Vam& &     &       &\Proj(\Vam)&      &   & &\Proj(\Vhp)& &  &  &\Vhp \\
&\rdTo&   &\ldTo_{\ra{\;\pip}}&    &      &   & &  &\rdTo_{\ra{\pim\,}}& &\ldTo\\
&     &S\hp&          &        &     &    &  &   &   &S\am &
\end{diagram}

\step4{Canonical twistor lines} We now reach the key point of the
construction.  Whereas any fibre $p^{-1}(x)$ of $p\colon \hat
Z=\Proj(\Lp^*\ds\Lm^*)\to \Sc$ has trivial normal bundle in $\hat Z$, its
image $\phi(p^{-1}(x))$, called a \emph{canonical twistor line}, has normal
bundle isomorphic to $\C^{2n}\tp\cO(1)$ in the blow-down $Z$. We thus obtain
our main result.

\begin{thm}\label{maintheorem} Let $(S,\Pi_c)$ be a c-projective manifold
of type $(1,1)$. Then for any complex line bundle $\cL$ with connection
$\nabla$ of type $(1,1)$, the holomorphic manifold $Z$ of
Definition~\textup{\ref{d:Z}} is the twistor space of a quaternionic manifold
$M$ with a quaternionic $S^1$ action having $S$ as a component of its fixed
points.  Furthermore, $S$ is a totally complex submanifold of $M$, with induced
c-projective structure $\Pi_c$, and a neighbourhood of $S$ in $M$ is
$S^1$-equivariantly diffeomorphic to a neighbourhood of the zero section in
$TS\tp(\Lm^*\tp\Lp)|_S$.
\end{thm}
\begin{proof}
\begin{bulletlist}
\item By Proposition \ref{p:manifold}, $Z$ is a holomorphic manifold with a
  holomorphic $S^1$ action.
\item By Corollary \ref{c:normal}, the canonical twistor lines form a family
  of projective lines in $Z$ with normal bundle isomorphic to
  $\C^{2n}\tp\cO(1)$.
\item By Proposition \ref{p:rs}, $\rs$ is an $S^1$-equivariant antiholomorphic
  involution of $Z$, the canonical twistor lines parametrized by $S\sub \Sc$
  are real, and $\rs$ has no fixed points.
\end{bulletlist}
\noindent Thus $Z$ is the twistor space of a quaternionic manifold $M$ with a
quaternionic $S^1$ action. By Proposition~\ref{p:cproj}, $S$ is a (maximal)
totally complex submanifold, with induced c-projective structure $\Pi_c$. The
$S^1$-equivariant diffeomorphism follows from Proposition~\ref{p:S1TNT}, and
hence $S$ is a component of the fixed point set of the $S^1$ action on $M$.
\end{proof}
\begin{defn} The construction of $Z$ and $M$ in Theorem~\ref{maintheorem} from
  $S$ and $\cL$ is called the \emph{quaternionic Feix--Kaledin construction}.
\end{defn}
It remains to understand when a quaternionic $4n$-manifold $(M,\qs)$ with a
quaternionic $S^1$ action arises in this way. For this note that at any fixed
point $x\in M$, the $S^1$ action induces a linear action on the $\symp(1)$
subalgebra $\qs_x\sub\gl(T_xM)$ preserving the bracket (or equivalently, the
inner product). If the action is trivial, we say $x$ is \emph{triholomorphic};
otherwise the action is generated by a positive multiple of $[J,\cdot]\in
\qs_x$ for some $J\in\qs_x$ with $J^2=-\id$ (this is a rotation fixing
$\spanu\{J\}\sub\qs_x$).

\begin{thm}\label{conversetheorem} Let $(M,\qs)$ be a quaternionic
$4n$-manifold with a quaternionic $S^1$ action whose fixed point set has a
connected component $S$ which is a submanifold of real dimension $2n$ with no
triholomorphic points. Then $S$ is totally complex, and a neighbourhood of $S$
in $M$ arises from the induced c-projective structure on $S$ via the
quaternionic Feix--Kaledin construction, for some complex line bundle $\cL$ on
$S$.
\end{thm}

\section{Background on projective geometries}\label{s:bg}

\subsection{Complexification of complex manifolds}\label{s:complexification}

Let $S$ be a real analytic manifold with complexification $(\Sc,\rs)$ where
$\rs$ is a real structure as in \S\ref{s:cpg} (cf.~also~e.g.~\cite{Biel}
or~\cite[p.66]{Manin}). If $\cE$ is a real-analytic vector bundle of rank $k$
over $S$, then in some connected neighbourhood of $S$ in $\Sc$, the transition
functions of $\cE$ have holomorphic extensions and hence we may construct a
holomorphic vector bundle $\cE\cx$ of complex rank $k$ over $\Sc$, with an
isomorphism $\rs^*\cE\cx\cong\overline{\cE\cx}$. The complexification $\cE\cx$
is not unique, but any two complexifications of $\cE$ are locally isomorphic
near $S$. Note that $T\Sc$ is a complexification of $TS$.

If $\cE$ is a complex vector bundle with real-analytic complex structure $I$,
then, locally near $S$, we may assume $I$ extends to $\cE\cx$, thus defining a
decomposition $\cE\cx=\cE\hp\ds\cE\am$ into the $\pm\iI$ eigenspaces of $I$
($\iI^2=-1$). In particular, if $\dim S=2n$ and $J$ is a real-analytic almost
complex structure on $S$, then the tangent bundle of $\Sc$ has a decomposition
\begin{equation*}
T\Sc=T\hp \Sc\ds T\am \Sc,
\end{equation*}
into $\pm\iI$ eigendistributions of $J$. These distributions are integrable if
and only if $J$ is an integrable complex structure, in which case $T\hp \Sc$
and $T\am \Sc$ define two transverse foliations, interchanged by $\rs$, called
the $(1,0)$ and $(0,1)$ foliations. We may (locally)
assume these foliations are regular, and hence define fibrations
\begin{diagram}[size=1.5em,nohug]
     &   & \Sc \\
    &\ldTo^{\pip}& & \rdTo^{\pim}\\
S\hp& &  & &S\am
\end{diagram}
from $\Sc$ to the leaf spaces $S\hp$ and $S\am$ of the $(0,1)$ and $(1,0)$
foliations respectively; the real structure $\rs$ then induces a
biholomorphism $\theta\colon \overline{S\am}\to S\hp$. We may further assume
that the projections $\pip$ and $\pim$ are jointly injective, defining an
embedding
\begin{equation*}
(\pip,\pim)\colon \Sc\into S\hp\times S\am.
\end{equation*}
Thus we may identify $\Sc$ with an open subset of $S\hp\times S\am$, where
$\rs$ is induced by $(x,\tx)\mapsto (\theta(\tx),\theta^{-1}(x))$,
so that $S$ is identified with the ``antidiagonal'' $\{(x,\theta^{-1}(x)):x\in
S\hp\}$, and $T\hp \Sc\cong TS\hp$, $T\am \Sc\cong TS\am$ are tangent to the
factors.

If $\cE\to S$ is a complex vector bundle with an integrable
$\overline{\partial}$-operator, then (locally near $S$) the latter defines a
trivialization of $\cE\hp$ along the leaves of $(0,1)$ foliation, and of
$\cE\am$ along the leaves of $(1,0)$ foliation. Thus we may write $\cE\hp$ and
$\cE\am$ as pullbacks by $\pip$ and $\pim$ of holomorphic vector bundles on
$S\hp$ and $S\am$ respectively.

In summary, a $2n$-manifold $S$ with an integrable complex structure $J$ has
an essentially canonical complexification: we may \emph{define} $S\hp$ to be
$S$ equipped with the holomorphic structure induced by $J$, and
$S\am=\overline{S\hp}$ (which has the holomorphic structure induced by $-J$)
so that the biholomorphism $\theta\colon \overline{S\am}\to S\hp$ is the
identity.

\begin{prop} If $S$ has an integrable complex structure, then $S\hp\times S\am$,
is a complexification of $S$, with $\rs(x,\tx)=(\tx,x)$, and any sufficiently
small complexification $\Sc$ of $S$ may be identified with a neighbourhood of
the \textup(anti\textup)diagonal in $S\hp\times S\am$.

A complex vector bundle $\cE\to S$ with an integrable
$\overline{\partial}$-operator defines holomorphic vector bundles $\cE\hp\to
S\hp$ and $\cE\am\to S\am$, where $\cE\am=\overline{\theta^*\cE\hp}$, and
\textup(omitting pullbacks by $\pip$ and $\pim$\textup) $\cE\hp\ds \cE\am\to
\Sc$ is a complexification of $\cE\to S$.
\end{prop}

Suppose that $D$ is a real-analytic affine connection on $S$, i.e., in
real-analytic coordinates, the connection $1$-forms of $D$ are given by
real-analytic functions. Then, using such coordinates, we can holomorphically
extend the connection $1$-forms near $S$ to obtain a holomorphic affine
connection $D\cx$ (i.e., it has holomorphic connection forms in holomorphic
coordinates) on some complexification $\Sc\sub S\hp\times S\am$.

Similarly if $\cE\to S$ admits a complex connection $\nabla$ which is
real-analytic (in a real-analytic trivialization of $\cE$) and compatible with
the holomorphic structure on $\cE$ (i.e.,
$\nabla^{0,1}=\overline{\partial}_{\cE}$), then locally we can complexify the
connection (by holomorphic extension of the connection forms) to obtain a
complexified connection $\nabla^c$ on $\cE\cx$.

\subsection{Blow ups and projective bundles}\label{s:pb-blow-up}

Recall that a map $p\colon\hat{\mfd}\to\mfd$ is called a \emph{blow-up} of a
  holomorphic manifold $\mfd$ along a submanifold $B$ with \emph{exceptional
    divisor} $\hat B\sub \hat\mfd$ (and $M$ is the \emph{blow-down} of
  $\hat\mfd$ along $p$) if
\begin{itemize}
\item $p|_{\hat B}\colon \hat B\to B$ is isomorphic to $\Proj(NB)\to
  B$, where $NB=T\mfd|_B/TB$, and
\item $p|_{\hat\mfd\setminus\hat B}\colon\hat\mfd\setminus\hat B\to
  \mfd\setminus B$ is a biholomorphism.
\end{itemize}
The prototypical example is the blow-up of a vector space $E$ at the origin,
given by the projection $\cO_E(-1)\into \Proj(E)\times E\to E$, where
$\cO_E(-1)$ is the tautological line subbundle of $\Proj(E)\times E$ over the
projectivization $\Proj(E)$ of $E$, whose fibre at $\ell\in\Proj(E)$ is
$\cO_E(-1)_\ell=\ell\leq E$. The exceptional divisor in this case is the zero
section of $\cO_E(-1)\to\Proj(E)$.

We make crucial use of the local geometry of blow-up and blow-down in our
constructions, so we summarize some key points here, as well as fixing
notation. First note that the inclusion of $\cO_E(-1)$ into $\Proj(E)\times E$
defines a section of the bundle $\Hom(\cO_E(-1),E)\to \Proj(E)$ with fibre
$\Hom(\cO_E(-1),E)_\ell=\Hom(\ell,E)$. Dually there is a canonical bundle map
$\Proj(E)\times E^*\to \cO_E(1):=\cO_E(-1)^*$ , sending $(\ell,\alpha)$ to
$\alpha|_\ell\in \ell^*$---hence a map from $E^*$ to the space of global
sections of $\cO_E(1)$. The image of this map is called the space
$\Aff(\cO_E(1))$ of \emph{affine sections} of $\cO_E(1)$ because of the
following standard fact.

\begin{obs}\label{obs:flatc} The bundle map $\Proj(E)\times E^*\to J^1\cO_E(1)$
induced by taking $1$-jets of affine sections is a bundle isomorphism. Hence
$J^1\cO_E(1)$ has a canonical flat \textup(indeed, trivial\textup) connection
whose parallel sections are $1$-jets of affine sections of $\cO_E(1)$, and
there is an exact sequence of bundles\textup:
\begin{equation}\label{euler}
0\to T^*\Proj(E)\tp\cO_E(1)\to \Proj(E)\times E^*\to \cO_E(1)\to 0.
\end{equation}
\end{obs}

To study more general blow-ups, we apply the above notions fibrewise to a
vector bundle $\cE \xrightarrow\pi \mfd$, which has a \emph{projectivization}
$\Proj(\cE)\to\mfd$ with $\Proj(\cE)_x=\Proj(\cE_x)$ for any $x\in\mfd$; we
further set $\cE\mult:= \cE\setminus\zers$, where $\zers$ is (the image of)
the zero section of $\cE$.  This an open subset of the \emph{fibrewise
  tautological bundle} $\cO_\cE(-1)\to \Proj(\cE)$ whose fibre over $\ell\in
\Proj(\cE)_x$ (for $x\in\mfd$) is $\ell\leq\cE_x$.  Then the projection from
$\cO_{\cE}(-1)$ to $\cE$ blows down the zero section of $\cO_\cE(-1)$ to the
zero section of $\cE$. There is however, an important detail we need to note.
\begin{remark} \label{tensorbylinebundle}
For any $1$-dimensional vector space $L$, $\Proj(E\tp L)$ is canonically
isomorphic to $\Proj(E)$, but $\cO_{E\tp L}(-1)= \cO_E(-1)\tp L$.  However, if
$\dim E = m+1$, then by taking the top exterior power of~\eqref{euler}, we
obtain that $\cO_E(m+1)\cong \Wedge^m T\Proj(E) \tp\Wedge^{m+1} E^*$.  (As
usual, for $k\in\Z$, $\cO_E(k):=\cO_E(1)^{\tp k}$ denotes the $k$th tensor
power.)
\end{remark}
Applying this remark fibrewise to $\cE\to\mfd$, we have that for any line
bundle $\cL\to\mfd$, $\Proj(\cE\tp\cL)$ is canonically isomorphic to
$\Proj(\cE)$, but $\cO_{\cE\tp\cL}(-1)\cong\cO_\cE(-1)\tp\pi^*\cL$.

We shall need one further variant of these constructions involving
\emph{projective completions} such as the projective line bundle
$\Proj(\cO\ds\cO_E(-1))\to\Proj(E)$.  This is a subbundle of
$\Proj(E)\times\Proj(\C\ds E)$, with fibre $\Proj(\C\ds\ell)\sub\Proj(\C\ds
E)$ over $\ell\in\Proj(E)$. Hence there is a blow-down map
$\Proj(\cO\ds\cO_E(-1))\to \Proj(\C\ds E)$ which is isomorphic to the
blow-down $\cO_E(-1)\to E$ on the complement of the section
$\Proj(\cO_E(-1))\cong\Proj(E)$.

\begin{obs}\label{normaltoline} In the blow-down $\Proj(\cO\ds\cO_E(-1))\to
\Proj(\C\ds E)$, the fibre $\Proj(\C\ds\ell)$ over $\ell\in\Proj(E)$ maps to
the corresponding projective line in $\Proj(\C\ds E)$, with normal bundle
$T\Proj(\C\ds E)|_{\Proj(\C\ds \ell)}/T\Proj(\C\ds\ell)\cong \cO_{\C\ds\ell}(1)
\tp E/\ell$.
\end{obs}
Here the normal bundle is identified by applying~\eqref{euler} to
$\Proj(\C\ds\ell)$ and $\Proj(\C\ds E)$.

\begin{remark}\label{rem:blowup} As this last example illustrates, blow-up
and blow-down are local to the submanifold or exceptional divisor. Hence
disconnected submanifolds and exceptional divisors can be blown up or down
componentwise. On the other hand, the blow-down of the inverse image of an
open subset $U\sub\Proj(E)$ in $\cO_E(-1)$ (for example) is the cone on $U$ in
$E$, which (for $U$ proper) is singular at the origin.
\end{remark}

\subsection{Cartan geometries}

Let $G$ be a real or complex Lie group, and $P$ a (closed) Lie subgroup, so
that $G/P$ is a (smooth or holomorphic) homogeneous space.  Let $M$ be a
(smooth or holomorphic) manifold with the same dimension as $G/P$.

\begin{defn} A \emph{Cartan connection of type} $(G,P)$ on $M$ is a
principal $G$-bundle $\cG\to M$, with a principal $G$-connection $\eta\colon
T\cG\to\g$ and a reduction $\iota\colon\cP\into\cG$ of structure group to
$P\leq G$ satisfying the following (open) \emph{Cartan condition}:
\begin{bulletlist}
\item the pullback $\iota^*\eta$ induces a bundle isomorphism of
$T\cP$ with $\cP\times\g$.
\end{bulletlist}
A manifold $M$ with a Cartan connection is called a \emph{Cartan geometry}.
Its \emph{Cartan bundle} is the bundle of homogeneous spaces
$\cC_M:=\cG/P\cong \cG\times_G (G/P)\cong \cG\times_P (G/P)$ over $M$.  The
principal connection $\eta$ on $\cG$ induces a connection on $\cC_M$, while
the reduction to $P$ equips $\cC_M$ with a \emph{tautological section}
$\tau\colon M\cong \cP/P\into \cG/P=\cC_M$.
\end{defn}
The model Cartan connection of type $(G,P)$ is the reduction $G\into
(G/P)\times G$ of principal bundles over $G/P$, with connection given by the
Maurer--Cartan form $\eta_G\colon TG\to \g$ of $G$. This is an isomorphism on
each tangent space, so the bundle map $T(G/P)\to G\times_P(\g/\p)$, induced by
the horizontal $1$-form $\eta_G+\p\in \Omega^1(G,\g/\p)^P$, is a bundle
isomorphism.

For a general Cartan geometry $M$ of type $(G,P)$, it follows that the
vertical bundle of $\cC_M$ is (isomorphic to) $\cG\times_G T(G/P)\cong
\cG\times_P(\g/\p)$, and the induced connection on $\cC_M$ is the $1$-form
$\eta_\cC\colon T\cC_M\to \cG\times_P(\g/\p)$ induced by the (horizontal,
$P$-equivariant) $1$-form $\eta+\p\colon T\cG\to\g/\p$.  Let
$\g_M=\cG\times_G\g\cong\cP\times_P\g$ and $\p_M=\cP\times_P \p$. Then the
covariant derivative $\eta_M:=\tau^*\eta_\cC\colon TM\to \cP\times_P(\g/\p)
\cong\g_M/\p_M$ of the tautological section $\tau$ is the $1$-form on $M$
induced by the pullback $\iota^*(\eta+\p)=\iota^*\eta+\p\colon T\cP\to
\g/\p$. The Cartan condition means (equivalently) that $\eta_M$ is a bundle
isomorphism.

The key idea behind Cartan connections is that if $\cD$ is flat, then in a
local trivialization $\cC_M$ by parallel sections over an open subset $U$, the
tautological section $\tau|_U\colon U\to \cC|_U\cong U\times G/P$ defines a
\emph{developing map} from $U$ to $G/P$: by the Cartan condition, these maps
are local diffeomorphisms, identifying $M$ locally with $G/P$. Since this
notion of development will be crucial to us, we establish it explicitly using
a linear representation of the Cartan connection described
in~\S\ref{s:pc-affine}.

\subsection{Projective parabolic geometries}\label{s:parabolic}

Smooth projective, c-projective and quaternionic manifolds are Cartan
geometries modelled on the projective spaces $\R P^n$, $\C P^n$ and $\HQ P^n$,
which are (real) homogeneous spaces for the projective general linear groups
$\PGL(n,\R)$, $\PGL(n,\C)$ and $\PGL(n,\HQ)$.  The corresponding holomorphic
Cartan geometries are modelled on complexifications of these varieties, namely
$\C P^n$, $\C P^n\times\C P^n$ and the grassmannian $\Gr_2(\C^{2(n+1)})$ of
two dimensional subspaces of $\C^{2(n+1)}$.

These Cartan geometries are examples (cf.~\cite{CaFr}) of \emph{parabolic
  geometries}~\cite{CS}: the model $G/P$ is a \emph{generalized flag variety},
with $G$ semisimple, and $\p$ a parabolic subalgebra of $\g$. This means that
the Killing perp $\p^\perp$ is a nilpotent ideal in $\p$---and in the above
examples, $\p^\perp$ is abelian. For such Cartan geometries, the isomorphism
$TM\cong \g_M/\p_M$ induces an isomorphism of
$T^*M\cong\p^\perp_M:=\cP\times_P \p^\perp$, the Lie bracket on $\g_M$ induces
a graded Lie bracket $\abrack{\cdot,\cdot}$ on $TM\ds (\p_M/\p_M^\perp)\ds
T^*M$, and so there is an algebraic bracket
\begin{equation*}
\abrack{\cdot,\cdot}\colon TM\times T^*M\to \p_M/\p_M^\perp\sub\gl(TM).
\end{equation*}
These geometries all admit an equivalence class $\Pi$ of torsion-free connections
$[D]$, where
\begin{equation*}
\tilde D\sim D \quad \Leftrightarrow\quad  \exists\, \gamma\in \Omega^1(M)\quad
\text{such that}\quad \tilde D_XY = D_XY+\abrack{X,\gamma}(Y)
\end{equation*}
for all vector fields $X,Y$. For projective, quaternionic and c-projective
manifolds, the bracket is defined explicitly in equations~\eqref{eq:proj},
\eqref{eq:quat} and~\eqref{eq:cproj} respectively.

If the curvature $R^D$ is viewed as a function of $D\in\Pi$, then its
derivative with respect to a $1$-form $\gamma$ is $\partial_\gamma
R^D=-\abrack{\id\wedge D\gamma}$, where $\partial_\gamma F(D)=\frac{\d}{\d t}
F(D+t\gamma)|_{t=0}$ and $\abrack{\id\wedge D\gamma}_{X,Y}
=\abrack{X,D_Y\gamma}-\abrack{Y,D_X\gamma}$.  One further feature of these
geometries is the existence of a ``normalized Ricci'' or ``Rho'' tensor
$\ri^D\in\Omega^1(M,T^*M)$ (a cotangent-valued $1$-form) such that
$\partial_\gamma\ri^D=-D\gamma$ and hence $\Wi:=\Ri^D-\abrack{\id\wedge\ri^D}$
is an invariant of the geometry (i.e., independent of $D\in \Pi$) called its
\emph{Weyl curvature}. It follows also that the \emph{Cotton--York curvature}
$C^D:=\d^D\ri^D$ satisfies $\partial_\gamma C^D=-\d^D
D\gamma+\abrack{\abrack{\id,\gamma}\wedge \ri^D}= -\abrack{\Wi,\gamma}$. In
particular, if the Weyl curvature vanishes, then the Cotton--York curvature is
an invariant.

Conversely, given an equivalence class $\Pi$ of torsion-free affine
connections on $M$, compatible with an appropriate reduction of the frame
bundle, the general theory of parabolic geometries~\cite{CS} constructs a
Cartan connection $\eta$ which is flat if and only if the Weyl and
Cotton--York curvatures vanish. We now discuss this for projective structures.

\subsection{Projective structures, affine sections and development}
\label{s:pc-affine}

On a projective space $\Proj(E)$, the trivialization
$J^1\cO_E(1)\cong\Proj(E)\times E^*$ of Observation~\ref{obs:flatc} may be
viewed as a linear representation of a flat Cartan connection. Its parallel
sections are $1$-jets of sections of $\cO_E(1)$ induced by linear forms on
$E$, which are affine functions in any affine chart. Globally, these are the
elements of the space $H^0(\Proj(E),\cO_E(1))$ of regular (or holomorphic)
sections.  Locally, these \emph{affine sections} of $\cO_E(1)$ are solutions
of a second order differential equation.  Projective structures generalize
this local description.

\begin{defn}\label{d:O(1)} Let $\mfd$ be a smooth or holomorphic $n$-manifold.
Then we denote by $\cO_\mfd(1)$ a (chosen) line bundle over $\mfd$ that
satisfies $\cO_\mfd(n+1):=\cO_\mfd(1)^{\tp(n+1)}\cong\Wedge^nT\mfd$.  We set
$\cO_\mfd(-1):=(\cO_\mfd(1))^*$.
\end{defn}

Let $\Pi_r$ be a projective structure on a manifold $\mfd$.  A choice of
$D\in\Pi_r$ gives a splitting of the $1$-jet sequence
\begin{equation*}
0\to T^*\mfd\tp\cO_\mfd(1)\to J^1\cO_\mfd(1)\to\cO_\mfd(1)\to 0,
\end{equation*}
i.e., an isomorphism $J^1\cO_\mfd(1)\cong\cO_\mfd(1)\ds
(T^*\mfd\tp\cO_\mfd(1))$ sending $j^1\ell$ to $(\ell,D\ell)$. For $n>1$, there
is also a normalized Ricci tensor $\ri^D$ associated to $D$, with
$\partial_\gamma \ri^D=-D\gamma$.

\begin{defn}\label{d:linproj} For any $D\in\Pi_r$, $\ell\in\cO_\mfd(1)$
and $\alpha\in T^*\mfd\tp\cO_\mfd(1)$, let $\vecD{\ell}{\alpha}=
j^1\ell-D\ell+\alpha$ (defined using a local extension of $\ell$) be the
element of $J^1\cO_\mfd(1)$ corresponding to $(\ell,\alpha)\in (\cO_\mfd(1)\ds
T^*\mfd\tp\cO_\mfd(1))$. Define a connection $\cD$ on $J^1\cO_\mfd(1)$ by
\begin{equation}\label{eq:linproj}
\cD_X\VecD{\ell}{\alpha}=\VecD{D_X\ell-\alpha(X)}{D_X\alpha+(\ri^D)_X\ell}.
\end{equation}
\end{defn}

\begin{prop}\label{p:affinesections}
The connection $\cD$ does not depend on the choice of $D\in\Pi_r$.
\end{prop}
\begin{proof} Since $\partial_\gamma D\ell=\gamma\ell$, we have
\begin{equation*}
\partial_{\gamma} \VecD{\ell}{\alpha}
=\partial_\gamma(j^1\ell-D\ell+\alpha)=-\gamma\ell=\VecD{0}{-\gamma \ell}.
\end{equation*}
Then by the Leibniz rule
\begin{equation*}
\partial_{\gamma}\VecD{D_X\ell-\alpha(X)}{D_X\alpha+(\ri^D)_X\ell}
= -\VecD{0}{\gamma (D_X\ell-\alpha(X))} +\VecD{\gamma(X)\ell}
{\abrack{X,\gamma}^r\cdot\alpha - D_X\gamma\, \ell}.
\end{equation*}
Since $\alpha$ is $\cO_\mfd(1)$-valued $1$-form,
$\abrack{X,\gamma}^r\cdot\alpha=-\alpha(X)\gamma$, and hence
\begin{equation*}
\partial_{\gamma}\VecD{D_X\ell-\alpha(X)}{D_X\alpha+(\ri^D)_X\ell}
=\VecD{\gamma(X)\ell}{-\gamma D_X\ell-(D_X\gamma)\ell} =\cD_X\VecD{0}{-\gamma
  \ell}.
\end{equation*}
Thus $\partial_\gamma\circ\cD=\cD\circ\partial_\gamma$ on $J^1\cO_\mfd(1)$, which
completes the proof.
\end{proof}
\begin{defn}\label{defaffinesections} A section $\ell$ of $\cO_\mfd(1)$
over $\mfd$ is called an \emph{affine section} if $j^1\ell$ is a
$\cD$-parallel section of $J^1\cO_\mfd(1)$.  Note that if
$\vecD{\ell}{\alpha}$ is parallel for $\cD$ then $\alpha=D\ell$, i.e.,
$\vecD{\ell}{\alpha}=j^1\ell$, and hence $D^2\ell+\ri^D \ell=0$. Thus
$\ell\mapsto j^1\ell$ is a bijection between affine sections of $\cO_\mfd(1)$
and $\cD$-parallel sections of $J^1\cO_\mfd(1)$.
\end{defn}

\begin{prop}\label{p:flatproj}
$\cD$ is flat iff $\Pi_r$ has vanishing Weyl and Cotton--York curvatures.
\end{prop}
\begin{proof} Choosing $D\in\Pi_r$ and computing the curvature of $\cD$
from~\eqref{eq:linproj}, we obtain
\begin{equation*}
\Ri^{\cD}_{X,Y}\VecD{\ell}{\alpha}=\VecD{0}{W_{X,Y}\cdot\alpha+C^D_{X,Y}\ell}
\end{equation*}
for all vector fields $X,Y$, where we use the fact that $\tr (W_{X,Y})=0$.
\end{proof}
If $n>2$ and $W=0$, the differential Bianchi identity implies that
$C^D=d^D\ri^D=0$, and so $\cD$ is flat if and only if the projective Weyl
curvature vanishes. For $n=2$, $W$ is identically zero, and so $\cD$ is flat
if and only if the projective Cotton--York curvature (which is a projective
invariant, also known as the Liouville tensor) vanishes.

\begin{remark}\label{rem:coupledproj} If $\cL$ is a line bundle with
connection $\nabla$ on a projective manifold $\mfd$, then we can define a
coupled (tensor product) connection $\cD^\nabla$ on $J^1\cO_\mfd(1)\tp\cL$,
and the map $\ell\tp u\mapsto (j^1\ell)\tp u+\ell\tp\nabla u$ similarly
defines a bijection between distinguished ``affine sections'' of
$\cO_\mfd(1)\tp\cL$ and $\cD^\nabla$-parallel sections of
$J^1\cO_\mfd(1)\tp\cL$.
\end{remark}

\subsection{C-projective structures and their foliations}\label{s:cproj}
Let $(S,J)$ be a complex manifold of complex dimension $n>1$, and let $\Pi_c$
be a real-analytic c-projective structure on $S$ (i.e., there is a
real-analytic connection in $\Pi^c$). Then we can extend real-analytic
connections in $\Pi_c$ to a complexification $\Sc$ of $(S,J)$ as
in~\S\ref{s:complexification}. Since $\abrack{\cdot,\cdot}$ depends only on
$J$, it extends to any such complexification, the following is immediate.

\begin{obs}\label{obs:ccproj} There is a complexification $(\Sc,\Pi_c\cx)$
of $(S,J,\Pi_c)$ such that the holomorphic connections in $\Pi_c\cx$ are
holomorphic extensions of connections in $\Pi_c$. The c-projective Weyl and
Cotton--York curvatures of $\Pi_c\cx$ are holomorphic extensions of
corresponding c-projective Weyl and Cotton--York curvatures of $\Pi_c$.
\end{obs}
\begin{prop}\label{p:indproj} A holomorphic c-projective structure $\Pi_c\cx$
on $\Sc\into S\hp \times S\am$ induces holomorphic projective structures on
the leaves of the $(1,0)$ and $(0,1)$ foliations.
\end{prop}
\begin{proof} Since $T\Sc=TS\hp \ds TS\am$, any connection in $\Pi_c\cx$
induces a connection on any leaf by restriction and projection. Now vectors
tangent to the $(1,0)$ and $(0,1)$ foliations are of the form $X+\iI JX$ and
$X-\iI JX$ respectively, and for any $1$-form $\gamma$ on $\Sc$,
\begin{align*}
\abrack{X+\iI JX,\gamma}^c(Y+\iI JY)&=\abrack{X+\iI JX,\gamma}^r(Y+\iI JY),\\
\abrack{X-\iI JX,\gamma}^c(Y-\iI JY)&=\abrack{X-\iI JX,\gamma}^r(Y-\iI JY).
\end{align*}
Hence c-projectively related connections on $\Sc$, after restriction to leaves
of the $(1,0)$ and $(0,1)$ foliations, are projectively related.
\end{proof}
\begin{remark}\label{rem:cgeod} Conversely the projective structures on the
leaves determine $\Pi_c$: for any $y\in\Sc$ and any affine connections $D$ and
$\tilde D$ on the leaves through $y$, there is a unique affine connection at
$y$ preserving the product structure and restricting to $D$ and $\tilde D$.
\end{remark}

Since the decomposition $T\Sc=TS\hp\ds TS\am$ is a holomorphic extension of
the type decomposition $TS\tp\C=T\hp S\ds T\am S$ on $S$, the decomposition
\begin{equation*}
\Wedge^2T^*\Sc= \Wedge^2T^*S\hp\ds(T^*S\hp\tp T^*S\am)\ds\Wedge^2 T^*S\am
\end{equation*}
is a holomorphic extension of the type decomposition $\Wedge^2T^*S\tp\C=
\Wedge^{2,0} \ds \Wedge^{1,1} \ds \Wedge^{0,2}$.
\begin{defn}\label{type}  We say that a c-projective structure $\Pi_c$ on
$(S,J)$ has \emph{type $(1,1)$} if its c-projective Weyl and Cotton--York
  curvatures have type $(1,1)$.
\end{defn}
\begin{prop}\label{p:flatindproj} A real-analytic c-projective structure of
type $(1,1)$ induces flat projective structures on the leaves of $(1,0)$ and
$(0,1)$ foliations in any complexification.
\end{prop}
\begin{proof} As the c-projective Weyl and Cotton--York have type $(1,1)$,
their holomorphic extensions have vanishing pullbacks, as $2$-forms, to any
leaf of the $(1,0)$ or $(0,1)$ foliation. However, due to the relation between
the algebraic brackets in the proof of Proposition~\ref{p:indproj}, these
pullbacks are the projective Weyl and Cotton--York curvatures of the leaves,
so the the induced projective structures are flat by
Proposition~\ref{p:flatproj}.
\end{proof}

We now discuss the line bundle $\cL\to S$ with connection $\nabla$; its
holomorphic extension $\nabla\cx$ to $\Sc$ provides line bundles with
connection along the $(1,0)$ and $(0,1)$ foliations which we use to twist the
projective Cartan connections along the leaves as in
Remark~\ref{rem:coupledproj}. To preserve flatness of the leafwise projective
structures, we require $\nabla\cx$ to be flat along leaves, i.e., $\nabla$ has
type $(1,1)$ curvature. In particular, $\nabla^{0,1}$ is a holomorphic
structure on $\cL$.

For a simply-connected projective manifold, a twist by a flat line bundle is
essentially trivial, corresponding to the ambiguity in
$\cO_E(1)\to\Proj(E)=\Proj(E\tp L)$ mentioned in
Remark~\ref{tensorbylinebundle}.  However, here we have two families of
projective leaves, and ambiguities in the choice of $\cO(1)$ along these
leaves which need not be compatible---and which we want to encode in the
$(1,1)$ curvature of $\nabla$. Thus, rather than simply taking
$\Lp=\pip^*\cO_S(1)$, where $\cO_S(m+1)=\Wedge^nTS\hp$, we first twist by
$\cL$ and take $\Lp=\pip^*(\cL\tp\cO_S(1))$. As mentioned in the introduction,
in this more general construction, it can happen that $\cL$ and $\cO_S(1)$ are
not globally defined on $S$, but $\Lp$ is. Indeed, as we shall see in
\S\ref{origFK}, in the Feix's original construction $\cL=\cO_S(-1)$ and $\Lp$
is trivial. Varying $\cL$ and $\nabla$ is thus a fundamental difference
between Feix's construction and ours, and leads to non-equivalent quaternionic
structures as we shall see e.g. in \S\ref{CG}.

\subsection{C-projective surfaces, projective curves and conformal
geometry}\label{s:csurf}

A complex structure $J$ on an oriented surface $S$ is the same data as a
conformal structure, and complex connections are conformal connections.
Torsion-free conformal (i.e., complex) connections on $S$ form an affine space
modelled on $1$-forms, i.e., a unique c-projective class on $S$.  However,
these data do not suffice to construct a Cartan connection modelled on the
flag variety $S^2\cong \C P^1$ for $\SO_0(3,1)\cong\mathrmsl{PSL}(2,\C)$, so
we need to modify the notion of a c-projective or conformal (M\"obius)
structure. Similarly a (real or holomorphic) projective curve $C$ has a unique
projective class of affine connections, but these do not determine the second
order \emph{Hill operator} on $\cO_C(1)$ whose kernel consists of the affine
sections.

Following~\cite{Cald}, we therefore require that $(S,J)$ is equipped with a
tracefree hessian operator (or \emph{M\"obius structure}), which is a second
order differential operator $\Delta\colon \Gamma L_S\to\Gamma \cS^2_0T^*S$, where
$L_S:=\cO_S(1)$ is a square root of $\Wedge^2 TS$, such that for some (hence
any) torsion-free connection $D$ there is a section $\ri^D_0$ of $\cS^2_0T^*S$
with
\begin{equation*}
\Delta(\ell)=\sym_0D^2\ell+\ri^D_0 \ell
\end{equation*}
for all sections $\ell$ of $L$. This allows us to construct a normalized Ricci
tensor $\ri^D$ with $\partial_{\gamma}\ri^D=-D\gamma$, which is the crucial
ingredient to build a Cartan connection.

Assuming $\Delta$ is real-analytic, it extends to a complexification $\Sc\into
S\hp\times S\am$, with
\begin{align*}
\cS^2_0T^*\Sc&=(T^*S\hp)^2 \ds (T^*S\am)^2,\\
(\ri^D_0)\cx&=(\ri^D_0)^{(2,0)}\ds (\ri^D_0)^{(0,2)}.
\end{align*}
We now define, as in \S\ref{s:pc-affine}--\S\ref{s:cproj}, a connection along
the leaves of the $(1,0)$ foliation by
\begin{equation*}
\cD\hp_Y\VecD{\ell}{\alpha}
=\VecD{D\hp_Y\ell-\alpha(Y)}{D\hp_Y\alpha+(\ri^D_0)^{(2,0)}_Y\ell},
\end{equation*}
where $\ell$ is a section of $\cO(1)$, $\alpha$ is an $\cO(1)$-valued
$(1,0)$-form and $Y$ is a $(1,0)$-vector field. As in
Proposition~\ref{p:affinesections}, $\cD\hp$ is independent of the choice of
$D$, and a similar construction applies along the leaves of the $(0,1)$
foliation.

\section{Quaternionic twistor theory}\label{s:qtt}

\subsection{Complexified quaternionic structures}

Let $Z$ be the twistor space of a quaternionic
manifold~\cite{Sal,HKLR,LeBrun,PP}, i.e., a holomorphic $(2n+1)$-manifold with
a real structure (antiholomorphic involution) $\rs\colon Z\to Z$, admitting a
\emph{twistor line} (a projective line which is holomorphically embedded in
$Z$ with normal bundle isomorphic to $\cO(1)\tp\C^{2n}$) which is real, i.e.,
$\rs$-invariant, and on which $\rs$ has no fixed points.

By Kodaira deformation theory~\cite{KK}, the moduli space of twistor lines in
$Z$ is a holomorphic $4n$-manifold $M\cx$, and there is an incidence relation
or correspondence
\begin{equation}
\begin{diagram}[height=1.3em,width=1.7em,nohug]
 &           &F_M\rlap{$\,:=\{(z,u)\in Z\times M\cx: z\in u\}$}\\
 &\ldTo_{\pi_Z}&  &\rdTo_{\pi_{M\cx}}\\
Z&           &  & &M\cx,
\end{diagram}
\end{equation}
where we identify $u\in M\cx$ with the corresponding twistor line
$u=\pi_Z(\pi_{M\cx}^{-1}(u))\sub Z$.  Thus $\pi_{M\cx}^{-1}(u)$ lifts $u\sub
Z$ to the incidence space $F_M$, which ``separates twistor lines'' (the fibres
are disjoint).  The normal bundles to twistor lines define a bundle $\Nb\to
F_M$ with fibre
\begin{equation*}
\Nb_{(z,u)}:=T_z Z/T_z u.
\end{equation*}
We then have~\cite{KK} that $T_uM\cx\cong H^0(u,\Nb|_u)$.

Locally over $M\cx$, we may decompose $\Nb$ (noncanonically) as
$\Nb=\pi_{M\cx}^*\cE\tp\pi_Z^*\cO_Z(1)$ where $\cE$ is a rank $2n$ bundle on
$M\cx$ and $\cO_Z(1)$ is a line bundle on $Z$ restricting to a dual
tautological bundle on each twistor line. Hence
\begin{equation*}
TM\cx\cong\cE\tp\cH,
\end{equation*}
where $\cH_u=H^0(u,\cO_Z(1)|_u)$, so that $F_M\to M\cx$ is canonically
isomorphic to $\Proj(\cH^*)\cong\Proj(\cH)$ (since $\cH$ has rank two), and we
have used that $\pi_{M\cx}^*\cE|_u=u\times\cE_u$. This tensor decomposition of
$TM\cx$ is the key structure carried by $M\cx$~\cite{PP,BE}, although
$\cE,\cH$ are only determined up to tensoring by mutually inverse line
bundles. The quaternionic connections on $M\cx$ are the tensor product
connections on $TM\cx= \cE\tp\cH$ which are torsion-free.

\begin{remark}\label{rem:EH}
We can restrict the freedom in $\cE$ and $\cH$ (locally) by requiring that
$\cO_{M\cx}(1):=\Wedge^2\cH=\Wedge^{2n}\cE$. This determines $\cH$ (and hence
$\cE$) up to a sign, so that $\cH\mult/\{\pm1\}$ is globally defined. Since
$\Wedge^{4n}TM\cx=(\Wedge^{2n}\cE)^2\tp(\Wedge^2\cH)^{2n}$, this convention
means equivalently that $\cO_{M\cx}(2n+2)=\Wedge^{4n}TM\cx$.  Taking top
exterior powers of
\[
0\to V\pi_{M\cx}\to \pi_Z^*TZ \to \Nb\to 0,
\]
using $V\pi_{M\cx}=\pi_{M\cx}^*(\Wedge^2\cH^*)\tp\pi_Z^*\cO_Z(1)$ and
$\Nb=\pi_{M\cx}^*\cE\tp\pi_Z^*\cO_Z(1)$, yields
\[
\pi_Z^*(\Wedge^{2n+1}TZ)
=\pi_{M\cx}^*(\Wedge^2\cH^*\tp\Wedge^{2n}\cE)\tp\pi_Z^*\cO_Z(2n+2).
\]
Thus a third equivalent formulation is that $\cO_Z(2n+2)=\Wedge^{2n+1}TZ$.
\end{remark}

\subsection{Null vectors, $\alpha$-submanifolds and projective structures}

We say a tangent vector to $M\cx$ is \emph{null} if it is decomposable in
$\cE\tp\cH$ and that a linear subspace of a tangent space is null if its
elements are. The fibre of $F_M$ over $z\in Z$ projects to a submanifold
$\alpha_z$ of $M\cx$ called an \emph{$\alpha$-submanifold}. Thus
$u\in\alpha_z$ iff $z\in u$, and then $T_u\alpha_z=\cE_u\tp\cO_Z(-1)_z$, so
that tangent spaces to $\alpha_z$ are null. Since the normal bundle to $u$ has
degree $1$, the twistor lines through $z\in u$ are determined by their tangent
space at $z$. Thus $\alpha_z$ is isomorphic to an open submanifold of
$\Proj(T_zZ)$, and has a canonical flat projective structure: any
$\Theta\in\Gr_{k+1}(T_z Z)$ parametrizes a $k$-dimensional projective (totally
geodesic) submanifold of $\alpha_z$ given by the twistor lines tangent to
$\Theta$ at $z$.

Any such null projective $k$-submanifold of $M\cx$ is determined by its
tangent space at a point $u\in M\cx$, which is a subspace of the form
$\theta\tp\ell \sub\cE_u\tp \cH_u=T_u M$ where $\theta$ is a $k$-dimensional
subspace of $\cE_u$ and $\ell$ is a $1$-dimensional subspace of $\cH_u$.  The
tangent lifts of null projective $k$-submanifolds thus foliate the subbundle
$\Gr_k(\cE)\times_{M\cx}\Proj(\cH)$ of null $k$-planes in
$\Gr_k(TM)\cap\Proj(\Wedge^k\cE\tp S^k\cH)\into \Proj(\Wedge^kTM)$ over the
grassmannian bundle $\Gr_{k+1}(TZ)$ as follows.
\begin{diagram}[height=1.5em,width=2.5em,nohug]
 &&\Gr_k(\cE)\times_{M\cx}\Proj(\cH)&\into&\Proj(\Wedge^k\cE\tp S^k\cH)
&\into&\Proj(\Wedge^k TM\cx)\\
&\ldTo&\dTo  & & \dTo & \ldTo(2,4)\\
\Gr_{k+1}(TZ)&   &\Proj(\cH) \\
\dTo&\ldTo_{\pi_Z}&  &\rdTo_{\pi_{M\cx}}\\
Z & & & & M\cx
\end{diagram}
For $k=1$, the geodesics of these projective structures are called \emph{null
geodesics} of $M\cx$. At the other extreme, when $k=2n-1$, $\Gr_{2n}(TZ)\cong
\Proj(T^*Z)$ and $\Gr_{2n-1}(\cE)\cong\Proj(\cE^*)$.

\begin{prop}\label{p:alpha} On any $\alpha$-submanifold $\alpha_z$ in a
complexified quaternionic manifold $M\cx$, any quaternionic connection $\qD$
induces an affine connection on $\alpha_z$ compatible with its canonical flat
projective structure.
\end{prop}
\begin{proof} Observe that $\pi_Z^{-1}(z)$ is the image of a section
of $\Proj(\cH)|_{\alpha_z}$ and if $h$ is a nonvanishing lift of this section
to $\cH|_{\alpha_z}$, then any vector (field) tangent to $\alpha_z$ have the
form $X=e\tp h$ for an element (or section) $e$ of $\cE|_{\alpha_z}$. Since
$\qD$ is torsion-free, and isomorphic to $\qD^\cE\tp\qD^\cH$, we have, for any
two null vector fields $X_1=e_1\tp h_1$ and $X_2=e_2\tp h_2$,
\begin{equation}\label{eq:qlb}
[X_1,X_2]= \qD^\cE_{X_1}e_2\tp h_2 - \qD^\cE_{X_2}e_1\tp h_1
+e_2\tp\qD^\cH_{X_1} h_2 - e_1\tp \qD^\cH_{X_2} h_1.
\end{equation}
If $h_1=h_2=h$, then $[X_1,X_2]$ is tangent to $\alpha_z$ for all $e_1,e_2$,
so $\qD^\cH_X$ preserves the span of $h$ for all $X$ tangent to $\alpha_z$.
Hence $\qD$ restricts to a (torsion-free) connection on $\alpha_z$.

It remains to show that $\qD$ preserves any projective hypersurface of
$\alpha_z$, i.e., the submanifold of twistor lines tangent to any hyperplane
in $T_zZ$. Such twistor lines generate a hypersurface $\cY$ in $Z$, and the
twistor lines in $\cY$ form a codimension two submanifold $Y$ of $M\cx$, with
conormal bundle $\eps\tp\cH^*$, where $\eps$ is a line subbundle of $\cE^*$
over $Y$. Now equation~\eqref{eq:qlb} implies that $\qD^\cE_X$ preserves
$\ker\eps$ along $Y$ for $X$ tangent to $Y$. Hence $Y\cap\alpha_z$ is totally
geodesic with respect to $\qD$.
\end{proof}

\section{Details and properties of the construction}\label{s:detail}

\subsection{The twistor space} \label{s:nb} We now fill in the remaining
details in the proof of Theorem~\ref{maintheorem}. First, we need to show that
$U\hp$ and $U\am$ can be chosen so that $Z$, constructed in
Definition~\ref{d:Z} is a twistor space with a holomorphic $S^1$ action.

\begin{prop}\label{p:manifold}  $Z$ is a complex manifold, with a holomorphic
vector field induced by scalar multiplication by $\lambda\in\C\mult$ in the
fibres of $\Vam$ and by $\lambda^{-1}$ in the fibres of $\Vhp$.
\end{prop}
\begin{proof} As $Z$ is obtained by gluing open subsets of the
manifolds $Z\am\sub\Vam$ and $Z\am\sub\Vhp$ by a relation intertwining the
action of $\lambda$ and $\lambda^{-1}$, it remains to show that $Z$ is
Hausdorff. So suppose $z\in Z\hp$ and $\tz\in Z\am$ with $[z]\neq [\tz]$ in
$Z$. If $z\in\im\php$ or $\tz\in\im\phm$ then we can replace it by the
corresponding point in $Z\am$ or $Z\hp$, which is distinct, hence separated,
from $\tz$ or $z$.  However, for $z\in U\hp$ and $\tz\in U\am$, the images of
$U\hp$ and $U\am$ are open, and separate $[z]$ and $[\tz]$ by
assumption~\eqref{eq:Zcond}.
\end{proof}

The construction of $Z$ from $\hat Z=\Proj(\Lp^*\ds\Lm^*)$ yields the diagram
\begin{equation}\label{diag1}
\begin{diagram}[height=1.5em,width=2.5em,nohug]
 &           &\hat Z\rlap{$=\Proj(\Lp^*\ds\Lm^*)$}\\
 &\ldTo(2,4)^{\phi}&\dTo  &\rdTo(2,4)^{p}\\
 &           &Z\times \Sc\\
 &\ldTo_{\pi_Z}&  &\rdTo_{\pi_{\Sc}}\\
Z&           &  & &\Sc.
\end{diagram}
\end{equation}
The induced (vertical) map $(\phi,p)\colon\hat Z\to Z\times \Sc$ is injective
and its image is the incidence relation $F_S\sub F_M$ for canonical twistor
lines: for $y\in\Sc$, we write $u(y):=\phi(p^{-1}(y))$ for the canonical
twistor line parametrized by $y$.

\begin{defn} The \emph{normal bundle} $\Nb$ on $\hat Z\cong F_S$ is the
bundle $\phi^*TZ /Vp$, where $Vp$ denotes the vertical bundle of $p\colon\hat
Z\to\Sc$, with fibre $\Nb_{(z,y)}=T_z Z/ T_z (u(y))$.
\end{defn}

\begin{prop}\label{p:normal} $\Nb=\Nb\hp\ds\Nb\am$, where
\begin{align*}
\Nb\hp&\cong p^*(TS\hp\tp\Lp^*)\tp\cO_{\Lp^*\ds\Lm^*}(1),\\
\Nb\am&\cong p^*(TS\am\tp\Lm^*)\tp\cO_{\Lp^*\ds\Lm^*}(1).
\end{align*}
\end{prop}
\begin{proof} For any $y=(x,\tx)\in \Sc$, we define $(n+1)$-dimensional
submanifolds of $Z$ by
\begin{align*}
\hat Z\hp_{\tx}&= Z\hp_{\tx}\cup\phm((\pim\circ p)^{-1}(\tx)),\\
\hat Z\am_{x} &= Z\am_{x}\cup\php((\pip\circ p)^{-1}(x)).
\end{align*}
By Remark~\ref{Propofphi}, these are well defined smooth submanifolds of $Z$,
and for any $y=(x,\tx)\in \Sc$, we have
\begin{equation*}
T\hat Z\hp_{\tx}|_{u(y)} + T\hat Z\am_{x}|_{u(y)}=TZ|_{u(y)}\quad\text{and}\quad
T\hat Z\hp_{\tx}|_{u(y)}\cap T\hat Z\am_{x}|_{u(y)}=Tu(y)
\end{equation*}
Hence $\Nb=\Nb\hp\ds\Nb\am$, where
\begin{equation*}
\Nb\hp_{(z,y)}=T_z\hat Z\hp_{\tx}/T_z u(y)\quad\text{and}\quad
\Nb\am_{(z,y)}=T_z\hat Z\am_{x}/T_z u(y).
\end{equation*}
The (canonical) identification of $\Nb\hp$ with
$p^*(TS\hp\tp\Lp^*)\tp\cO_{\Lp^*\ds\Lm^*}(1)$ follows easily from
Observation~\ref{normaltoline}, as $\hat Z\hp_{\tx}$ is a blow-down along the
zero section of the projective bundle $p^{-1}(\pim^{-1}(\tx))\sub
\Proj(\Lp^*\ds\Lm^*)$ over $\pim^{-1}(\tx)\rInto^{\pip} S\hp$, and
$\Vhp_{\tilde x}/u(y)\cong T_x S\hp\tp(\Lp^*\tp\Lm)_y$. A similar argument
identifies $\Nb\am$.
\end{proof}

We next construct the real structure on $Z$. By definition the holomorphic
line bundles $\overline{\Lm}\to \overline{S\am}$ and $\Lp\to S\hp$ are
isomorphic, and we denote the biholomorphisms $\overline{S\am}\to S\hp$ and
$\overline{\Lm}\to \Lp$ by $\theta$. The real structure $\rs$ on $\Sc\into
S\hp\times S\am$ sends $(x,\tx)$ to $(\theta(\tx),\theta^{-1}(x))$.  We lift
this real structure to $\hat Z=\Proj(\Lp^*\ds\Lm^*)$ by defining
$\rs([\sigma,\tilde\sigma])=[\tilde\sigma\circ\theta^{-1},-\sigma\circ\theta]$,
where the minus sign ensures $\rs$ has no fixed points. Since
$\rs(\zers)=\infs$, $\rs$ maps $\Lp\tp\Lm^*$ to $\Lp^*\tp\Lm$.  Since the
leafwise connections $\cD^\nabla$ are (by construction) related by $\theta$,
$\rs$ induces an antiholomorphic isomorphisms, also denoted $\rs$, between
$\Vam$ and $\Vhp$, with $\rs\circ\php=\phm\circ\rs$ and
$\rs\circ\phm=\php\circ\rs$. We further observe (again by construction) that
for any $v\in\Vam$,
\begin{equation*}
\rs(\lambda\cdot v)=\rs(\lambda v) = \overline\lambda\rs(v)
=\overline\lambda^{-1}\cdot \rs(v),
\end{equation*}
where $\cdot$ denotes the $\C\mult$ action. Thus $\rs$ intertwines the
$S^1$ actions on $\Vam$ and $\Vhp$.
\begin{prop} \label{p:rs} We may choose $U\am$ and $U\hp$ so that $Z\am$
and $Z\hp$ are $S^1$-invariant with $\rs(Z\am)=Z\hp$. Then $\rs$ induces an
$S^1$-invariant antiholomorphic involution of $Z$ with no fixed points on any
real \textup($\rs$-invariant\textup) canonical twistor line.
\end{prop}
\begin{proof} Take $U\am$ to be a sufficiently small $S^1$-invariant
neighbourhood of the zero section in $\Vam$ so that
$\phm^{-1}(U\am)\cap\rs(\phm^{-1}(U\am))=\emptyset$. Now set
$U\hp=\rs(U\am)$.  The real canonical twistor lines are the images of the
fibres of $p$ over the real submanifold $S\sub\Sc$. Since
$\rs\circ\phi=\phi\circ\rs$, $\rs$ has no fixed points on any such twistor
line.
\end{proof}
\begin{cor}\label{c:normal} $Z$ is a twistor space, and for any canonical
twistor line $u=u(y)$ \textup(with normal bundle
$\Nb|_u\cong\Nb\hp|_u\ds\Nb\am|_u$ isomorphic to $\C^{2n}\tp\cO(1)$\textup),
\begin{align*}
H^0(u,\Nb\hp|_u)&= (TS\hp\tp\Lp^*)_y\tp (\Lp\ds\Lm)_y\\
H^0(u,\Nb\am|_u)&= (TS\am\tp\Lm^*)_y\tp (\Lp\ds\Lm)_y.
\end{align*}
\end{cor}

\subsection{The quaternionic manifold} By Corollary~\ref{c:normal}
and~\cite{BE}, the moduli space of twistor lines in $Z$ is a complexified
quaternionic manifold $M\cx$ with $TM\cx=\cE\tp\cH$, where
\begin{equation}\label{eq:TM}\begin{split}
\cE|_{\Sc} &= (TS\hp\tp\Lp^*)\ds(TS\am\tp\Lm^*), \qquad \cH|_{\Sc}=\Lp\ds\Lm\\
TM\cx|_{\Sc} &=TS\hp\ds TS\am \ds (TS\hp\tp\Lp^*\tp\Lm)\ds (TS\am\tp\Lp\tp\Lm^*).
\end{split}\end{equation}
Note that in this decomposition, the terms $TS\hp\ds TS\am$ correspond to the
tangent space to the submanifold $\Sc$ of $M\cx$. Furthermore, the moduli
space of real twistor lines is a real quaternionic manifold $M$ in $M\cx$
containing $S$~\cite{PP}. Since the $S^1$ action on $Z$ is generated by a
holomorphic vector field, whose local flow maps twistor lines to twistor
lines, it induces an $S^1$ action on $M\cx$, preserving $M$, and fixing $\Sc$
pointwise.

\begin{prop}\label{p:cproj} $S$ is a maximal totally complex submanifold
of $M$, and the induced c-projective structure via
Theorem~\textup{\ref{thm:totallycomplex}} is the original c-projective
structure $\Pi_c$ on $S$.
\end{prop}
\begin{proof} By~\cite{BE,PP}, $\qs\sub\gl(TM)$ is isomorphic to the bundle
of real tracefree endomorphisms of $\cH|_M$. The real endomorphisms of $\cH|_S
= (\Lp\ds\Lm)|_S$ (see~\eqref{eq:TM}) are those commuting with its
quaternionic structure $(\sigma,\tilde\sigma)\mapsto
(\tilde\sigma\circ\theta^{-1},\sigma\circ\theta)$. In particular
\begin{equation*}
J=\begin{pmatrix}
\iI&0\\ 0&-\iI
\end{pmatrix}
\end{equation*}
is a section of $\qs$, preserving $TS$, and inducing the original complex
structure on $S$. The bundle $J^\perp$ consists of endomorphisms of $\cH$
of the form
\begin{equation*}
I_s=\begin{pmatrix}
0&-s^{-1}\\ s&0
\end{pmatrix}.
\end{equation*}
where $s$ is a unit section of $(\Lp^*\tp\Lm)|_S$. Clearly the induced
endomorphisms of $TM$ maps $TS$ into $(TS\hp\tp\Lp^*\tp\Lm)\ds
(TS\am\tp\Lp\tp\Lm^*)$. Thus $(S,J)$ is a maximal totally complex submanifold
of $(M,Q)$.

By Remark~\ref{rem:cgeod} the original and induced c-projective structures on
$S$ are uniquely determined by the corresponding families of holomorphic flat
projective structures on the leaves of the $(1,0)$ and $(0,1)$ foliations of
$\Sc$. For $x\in S\hp$, the original flat projective structure on
$\pip^{-1}(x)$ has a development into $\Proj(\Vam_x)\sub\Proj(T_z Z)$, where
$z$ is the zero vector in $\Vam_x$. Hence $\pip^{-1}(x)$ is a projective
submanifold of the $\alpha$-submanifold corresponding to $z$ (with its
canonical projective structure). Hence by Proposition~\ref{p:alpha}, any
quaternionic connection on $M\cx$ induces a connection on $\pip^{-1}(x)$
compatible with its original projective structure.
\end{proof}

\begin{prop}\label{p:S1TNT} Locally near $S$, $M$ is $S^1$-equivariantly
diffeomorphic to a neighbourhood of the zero section of $TS\tp\ul$, where $\ul
=(\Lp\tp\Lm^*)|_S$ is unitary.
\end{prop}
\begin{proof} By~\eqref{eq:TM}, the normal bundle to $S$ in $M$ is the real
  part of $(TS\hp\tp\Lp^*\tp\Lm)\ds (TS\am\tp\Lp\tp\Lm^*)$. The result now
  follows by the equivariant tubular neighbourhood theorem.
\end{proof}

This completes the details needed for the proof of Theorem~\ref{maintheorem}.

\subsection{Proof of Theorem~\ref{conversetheorem}}\label{ctproof}

Let $(M,\qs)$ be a quaternionic $4n$-manifold with a quaternionic $S^1$ action
whose fixed point set has a connected component $S$ which is a submanifold of
real dimension $2n$ with no triholomorphic points.

If $J$ is the section of $\qs|_S$ generating the infinitesimal $S^1$ action,
then $(TM|_S,J)$ decomposes into weight spaces for the action, with zero
weight space $TS$. Thus $TS$ is $J$-invariant, and for any $I\in J^\perp$,
$ITS$ is a nonzero weight space, complementary to $TS$ in $TM$.  It follows
that $S$ is a (maximal) totally complex submanifold of $M$. By restricting to
a neighbourhood of $S$ in $M$, we may assume that the $S^1$ action has no
other fixed points. It thus lifts to a holomorphic $S^1$ action on the twistor
space $Z\to M$, generated by a holomorphic vector field transverse to the
fibres over $M\setminus S$, tangent to the fibres over $S$, and vanishing
(only) along the sections $\pm J$ of $Z|_S$, denoted $S\hp$ and $S\am$. Let
$\phi\colon\hat Z\to Z$ be the blow-up of $Z$ along $S\hp\cup S\am$, with
exceptional divisor $\zers\cup\infs$, where $\zers$ and $\infs$ are the
projective normal bundles in $Z$ of $S\hp$ and $S\am$ respectively.  The real
structure on $Z$ (induced by $-\id$ on $\qs$) interchanges $S\hp$ and $S\am$,
and induces a fibre-preserving real structure on $\smash{\hat Z}$
interchanging $\zers$ and $\infs$.

The proper transform in $\hat Z$ of any fibre of $Z|_S$ is a rational curve
with trivial normal bundle meeting both $\zers$ and $\infs$.  Thus
$\phi^{-1}(Z|_S)$ has a neighbourhood foliated by a $2n$-dimensional moduli
space $\Sc$ of rational curves with trivial normal bundle. Each such curve
meets $\zers$ and $\infs$ in unique points, and projects to a twistor line in
$Z$ meeting $S\hp$ and $S\am$ in unique points. The induced map $\Sc\to
S\hp\times S\am$ is an immersion along the proper transforms of the fibres of
$Z|_S$, hence an open embedding in a neighbourhood.  Thus we may assume $\hat
Z$ is a $\C\Proj^1$-bundle over a complex $2n$-manifold $\Sc$, which embeds as
an open subbundle of $\zers\to S\hp$ and $\infs\to S\am$, and as an open
neighbourhood $\Sc$ of the diagonal in $S\hp\times S\am$. By
Lemma~\ref{lem:totallycomplex} and Proposition~\ref{p:alpha}, the induced
c-projective structure on $S$ has c-projective curvature of type $(1,1)$: in
the complexified c-projective structure on $\Sc$, the fibres over $S\hp$ and
$S\am$ are projectively-flat.

The holomorphic $S^1$ action on $Z$ has a single nontrivial weight space at
each point of $S\hp\cup S\am$ (the normal bundle to $S$ in $M$ has the same
weight as the normal bundle to $S\hp$ or $S\am$ in $Z|_S$). Hence it acts by
scalar multiplication on the normal bundles $\Vam$ to $S\hp$ in $Z$, and $\Vhp$
to $S\am$ in $Z$.  In particular, the $S^1$ action is trivial on the
projectivizations of $\Vam$ and $\Vhp$, i.e., the lifted action on $\hat Z$
fixes $\zers\cup\infs$ pointwise.  Thus $\hat Z\setminus(\zers\cup\infs)$ is a
holomorphic principal $\C\mult$-bundle over $\Sc$, with associated
$\C\Proj^1$-bundle $\hat Z$. The associate (dual) line bundles are subbundles of
the pullbacks of $\Vam$ and $\Vhp$ to $\Sc$, which thus have trivial Cartan
connections along the fibres over $S\hp$ and $S\am$ respectively.  Unravelling
the constructions in~\S\ref{s:cproj}, these are twists of the Cartan
connections induced by the c-projective structure by dual and conjugate line
bundles which are flat along the fibres over $S\hp$ and $S\am$; we deduce that
these twists come from a complex line bundle $\cL\to S$ with both a holomorphic
and an antiholomorphic structure, hence a (Chern) connection with curvature of
type $(1,1)$. We now have reconstructed the data for the quaternionic
Feix--Kaledin construction of $Z$ as a blow-down on $\hat Z$, and hence of (a
neighbourhood of $S$ in) $M$. \qed

\section{Examples and applications}\label{s:examap}

\subsection{Complex grassmannians}\label{CG} In~\cite{Wolf}, J.~Wolf classified
the totally complex submanifolds of quaternionic symmetric spaces fixed by a
circle action. These provide many examples of the quaternionic Feix--Kaledin
construction which are not (even locally) hypercomplex. We focus on the the
quaternionic symmetric spaces isomorphic (for some $n\geq 1$) to
$\Gr_2(\C^{n+2})$, the complex grassmannian of $2$-dimensional subspaces of
$\C^{n+2}$.  The twistor space $Z$ is the flag manifold $F_{1,n+1}(\C^{n+2})$
of pairs $B\sub W\sub \C^{n+2}$ with $\dim B=1$ and $\dim W=n+1$. The standard
hermitian inner product $\langle\cdot,\cdot\rangle$ on $\C^{n+2}$ defines a
real structure on $Z$, sending the flag $B\sub W$ to $W^\perp\sub B^\perp$. It
also defines an antiholomorphic diffeomorphism between $\Gr_2(\C^n)$ with
$\Gr_n(\C^{n+2})$, and it is convenient to identify the quaternionic manifold
$M$ with the graph of this map in $\Gr_2(\C^{n+2})\times \Gr_n(\C^{n+2})$. In
these terms the twistor projection from $Z$ to $M$, whose fibres are the real
twistor lines, sends $B\sub W$ to the pair $(B\ds W^\perp, B^\perp\cap W)$ in
$M$.

The space of all twistor lines in $Z$ is the holomorphic (i.e., complexified)
quaternionic manifold $M\cx\cong\{(U,V)\in\Gr_2(\C^{n+2})\times\Gr_n(\C^{n+2}):
\C^{n+2}=U\ds V\}$:
\begin{bulletlist}
\item the flags $B\sub W$ on the twistor line corresponding to $(U,V)\in M\cx$
  have $B\sub U$ and $V\sub W$, so that $B=U\cap W$ and $W=V+B$;
\item this twistor line is canonically isomorphic to $\Proj(U)\cong
  \Proj(\C^{n+2}/V)$;
\item also $\cO_U(-1)\cong \cO_{\C^{n+2}/V}(-1)$ via the map sending $b\in U$ to
  $b+V$ in $\C^{n+2}/V$.
\end{bulletlist}

A fixed decomposition $\C^{n+2}=A\ds \tilde A$, with $\dim A=1$ and $\dim
\tilde A=n+1$, determines a submanifold $\Sc=\{(U,V)\in M\cx:A\sub U,\; V\sub
\tilde A\}$ of $M\cx$:
\begin{bulletlist}
\item $(U,V)\mapsto (U/A,V)$ embeds $\Sc$ as an open subset of
  $\Proj(\C^{n+2}/A)\times\Gr_n(\tilde A)$;
\item the fibre of $\Sc$ over $V\sub\tilde A$ is isomorphic to the affine
  space $\Proj(\C^{n+2}/A)\setminus\Proj((V\oplus A)/A)$ and similarly for the
  fibre over $U\supseteq A$;
\item $\Proj(\C^{n+2}/A)\cong\Proj(\tilde A)$ may be identified with
  $S\hp=\{B\sub\tilde A:\dim B=1\}\sub Z$, and, similarly, $\Gr_n(\tilde
  A)\cong\Gr_n(\C^{n+2}/A)$ with $S\am=\{A\sub W:\dim W=n+1\}\sub Z$.
\item $\Gr_n(\tilde A)\cong \Proj(\tilde A^*)$ is the dual projective space to
  $\Proj(\C^{n+2}/A)\cong\Proj(\tilde A)$, and for any $(U,V)\in\Sc$, the
  corresponding tautological lines $(\tilde A/V)^*\cong V^0\sub \tilde A^*$ and
  $U/A\cong U\cap \tilde A$ are canonically dual to each other.
\end{bulletlist}
If $\tilde A=A^\perp$ then the real points in $\Sc\sub M\cx$ form a maximal
totally complex submanifold $S\sub M$ fixed by an $S^1$ action, and $S\hp$,
$S\am$ are lifts of $S$ to $Z$ with respect to the induced complex structures
$\pm J$ on $S$. Hence Theorem~\ref{conversetheorem} applies.

Following the proof in~\S\ref{ctproof}, let $\hat Z$ be the blow-up of $Z$
along $S\hp\cup S\am$.  The fibre of $\hat Z\to \Sc$ over $(U,V)$ is
$\Proj(U)\cong\Proj(\C^{n+2}/V)$, and the natural map to $Z$ is a
biholomorphism over $(B\sub W)\in Z$ unless $B=A$ or $W=\tilde A$, which are
the ``zero'' and ``infinity'' sections $\zers$ and $\infs$ of $\hat Z\to \Sc$,
mapping to $S\hp$ and $S\am$ respectively. Identifying $\Sc=\Proj(\tilde
A)\times \Proj(\tilde A^*)$, $\hat Z\cong\Proj(\cO_{\tilde
  A}(-1)\ds\cO)|_{\Sc}\cong \Proj(\cO\oplus \cO_{\tilde A^*}(-1))|_{\Sc}$.

We now set $\tilde A=A^\perp$ and identify $\Proj(\tilde A^*)$ with
$\overline{\Proj(A^\perp)}$ using the real structure; thus $\Sc$ is the open
subset $\{([\ell],[w])\in \Proj(A^\perp)\times \overline{\Proj(A^\perp)}:
\langle\ell,w\rangle\neq 0\}$, and the hermitian metric induces a pairing of
the tautological line bundles over $\Proj(A^\perp)$ and
$\overline{\Proj(A^\perp)}$, i.e., a nonvanishing section of $\cO(1,1)\to\Sc$.
On the (anti-)diagonal $S$ in $\Proj(A^\perp)\times\overline{\Proj(A^\perp)}$,
this section may be viewed as a hermitian metric on $\cO(-1)\to S$.

Locally, $\cO(-1)\to S$ has a square root $\cL=\cO(-\frac 12)$, and the
trivialization of $\cO(1,1)$ identifies $\cO(1,0)$ with $\cO(\frac
12,-\frac12)$. Thus we have the following result.
\begin{prop} Let $\Pi_c$ be the flat c-projective structure on $S$ and
let $\cL=\cO(-\frac12)$ \textup(defined over any open subset of
$S$\textup). The standard hermitian metric on $\C^{n+2}$ induces hermitian
metric on $\cL$ with Chern connection $\nabla$.  Then $Z$ and $M$ are obtained
from the quaternionic Feix--Kaledin construction applied to these data.
\end{prop}

Note that this example demonstrates the crucial role played by the twist: in
the motivating example (see Section \ref{motex}) we have shown that the flat
c-projective structure on $\C\Proj^n$ with trivial line bundle $\cL$ yield the
flat quaternionic model.

\subsection{The four-dimensional case and Einstein--Weyl spaces}\label{s:4d}

In four dimensions, a quaternionic manifold $(M,\qs)$ is a self-dual conformal
manifold.  LeBrun~\cite{Le} studied quotients of self-dual manifolds by a
class of $S^1$ actions which he called ``docile''; these include
\emph{semi-free} $S^1$ actions (whose stabilizers are either trivial or the
whole group), for which one of his results specializes as follows.

\begin{lemma}[\cite{Le}]\label{lemmaLe} Let $(M,g)$ be a self-dual manifold
with a semi-free $S^1$ action whose fixed point set is a nonempty surface $S$.
Let $B$ be a maximal smooth manifold \textup(without boundary\textup) in
$Y=M/S^1$. Then the Einstein--Weyl structure~\cite{Hit} $D$ on $B$ defined by
the Jones--Tod correspondence~\cite{JT} has $S$ as an asymptotically
hyperbolic end.
\end{lemma}
This means that $D$ is asymptotic (in a precise sense~\cite{Le}) to the
Levi-Civita connection of the hyperbolic metric in a punctured neighbourhood
of the image of $S$ in $Y$.
\begin{prop} The quotient by the $S^1$ action of the self-dual
conformal $4$-manifold obtained by the quaternionic Feix--Kaledin construction
is Einstein--Weyl with $S$ as an asymptotically hyperbolic end.
\end{prop}
\begin{proof} The $S^1$ action is induced by a holomorphic vector field on
the twistor space, which implies that it is conformal (see for example
\cite{JT}). It is also clearly semi-free and the zero section is the fixed
point set, which by Lemma~\ref{lemmaLe} completes the proof.
\end{proof}

There are special features of the quaternionic Feix--Kaledin construction of
$(M,\qs)$ from a surface $S$ with a c-projective structure.  As discussed in
\S\ref{s:csurf}, such a surface $S$ carries more data than $(J,\Pi_c)$. In the
approach discussed there, the additional data is a second order
operator~\cite{Cald}. Alternatively, one can characterize the Cartan
connection on $S$ or $\Sc$ explicitly. Following~\cite{Bor,BuCa}, we now
consider the latter approach (on $\Sc$).

A \emph{conformal Cartan connection} $(\cV,\Lam,\cD)$ on a holomorphic
surface $\Sc$ consists of:
\begin{bulletlist}
\item a rank $4$ holomorphic vector bundle $\cV\to \Sc$ with inner product
  $\langle,\rangle$;
\item a null line subbundle $\Lam\subset V$;
\item a linear metric connection $\cD$ satisfying the Cartan condition, that
  $\cD|_{\Lam} \mod \Lam$ is an isomorphism from $T\Sc\tp\Lam$ to
  $\Lam^{\perp}/\Lam$.
\end{bulletlist}
The Cartan condition implies that $T\Sc$ carries a conformal structure. We may
suppose that $\Sc\into S\hp\times S\am$ where the the leaves of the $(1,0)$
and $(0,1)$ foliations are the null curves of the conformal structure; we then
write $\Lam^\perp=U^++U^-$, where $U^+\cap U^-=\Lam$ and $\cD^{1,0}\Lam\sub
U^+$ and $\cD^{0,1}\Lam\sub U^-$. Observe that $\cD^{1,0}$ and $\cD^{0,1}$ are
flat connections, on $U^+$ and $U^-$ respectively, along the curves of the
$(1,0)$ and $(0,1)$ foliations respectively.

In~\cite{Bor}, the first author constructed a minitwistor space~\cite{Hit} of
an asymptotically hyperbolic Einstein--Weyl manifold $B$ from a conformal
Cartan connection by lifting the curves of $(1,0)$ and $(0,1)$ foliations to
$\Proj(U^+)$ and $\Proj(U^-)$ respectively, and gluing together the leaf
spaces. We now relate this approach to the quaternionic Feix--Kaledin
construction. The work of~\cite{Bor} already shows that $B$ is a quotient of a
self-dual $4$-manifold $M$ with an $S^1$ action, whose twistor space $Z$ is also
constructed explicitly there. Hence it suffices to establish the following.
\begin{prop} The construction of the twistor space in~\cite{Bor} from $S$
coincides with the quaternionic Feix--Kaledin construction given here.
\end{prop}
\begin{proof} The inner product on $\cV$ induces a duality between $U^+$
and $\cV/U^+$, with respect to which $\cD^{1,0}$ induces dual connections
along the curves of the $(1,0)$ foliation. We thus have isomorphisms
\begin{diagram}[size=1.5em]
0&\rTo&T^*S\hp \tp \cV/\Lam^\perp&\rTo &
J^1(\cV/\Lam^\perp)&\rTo& \cV/\Lam^\perp&\rTo& 0\\
&&\dTo&&\dTo&&\dEq&&\\
0&\rTo&\Lam^\perp/U^+&\rTo& \cV/U^+&\rTo& \cV/\Lam^\perp&\rTo& 0,
\end{diagram}
and similarly for $\cD^{0,1}$ on $U^-$ and $\cV/U^-$ along the $(0,1)$
foliation.

As explained in~\cite{Bor}, we may also suppose that $\Lam=\Lmp\tp\Lmm$, with
$\Lmp$ and $\Lmm$ trivial along the $(1,0)$ and $(0,1)$ foliations
respectively. The bundles $\tilde{U}^+:= U^+\tp(\Lmp)^{-2}$ and $\tilde{U}^-:=
U^-\tp(\Lmm)^{-2}$ have induced flat connections along the $(1,0)$ and $(0,1)$
foliations respectively, dual to $(\cV/U^-)\tp(\Lmp)^2$ and
$(\cV/U^+)\tp(\Lmm)^2$. Hence, along the null curves, the spaces $\cV^\pm$ of
parallel sections of $\tilde U^\pm$ are dual to spaces of affine sections of
$(\cV/\Lam^\perp)\tp(\Lam^\pm)^2\cong \Lam_\pm\tp(\Lam_\mp)^*$. Hence the
construction in~\cite{Bor} reduces to the one herein by taking $\Lmp=\Lm^*$ and
$\Lmm=\Lp^*$.
\end{proof}

The link with conformal Cartan connections elucidates the role of the
connection $\nabla$ on $\cL\to S$: any conformal Cartan connection over $S$,
is up to isomorphism, the twist of the normal Cartan connection (induced by a
M\"obius structure~\cite{Cald}) by such a connection $\nabla$. The
construction of the Einstein--Weyl manifold $B$ as an $S^1$-quotient equips it
with a distinguished gauge (or abelian monopole)~\cite{JT}.  Since
$\Proj(\cE\tp\cL)=\Proj(\cE)$ for any line bundle $\cL$ and vector bundle
$\cE$, the construction of the minitwistor space from $\Proj(\cV^+)$ and
$\Proj(\cV^-)$ does not depend on $(\cL,\nabla)$. We thus have a gauge for
each such choice.

\subsection{The hypercomplex and hyperk\"ahler cases}\label{origFK}

The line bundles $\Lp\to S\hp$ and $\Lm\to S\am$ which provide the input to
the quaternionic Feix--Kaledin construction are twists of the line bundle
$\cO_S(1)$, over a c-projective manifold $S$ with c-projective curvature of
type $(1,1)$, by a connection $\nabla$ on a complex line bundle $\cL\to S$
with curvature of type $(1,1)$. When $\cO_S(1)$ itself admits such a
connection, we can take $\cL=\cO_S(-1)$, so that $\Lp\to S\hp$ and $\Lm\to
S\am$ are trivial bundles.

\begin{prop}\label{propfeix} If the c-projective structure $\Pi_c$ on $S$
admits a real-analytic connection $D$ with curvature of type $(1,1)$, and
$\nabla$ is the induced connection on $\cL=\cO_S(-1)$, then the quaternionic
manifold $M$ of Theorem~\textup{\ref{maintheorem}} is hypercomplex, and is
the hypercomplex manifold constructed by Feix~\cite{Feix2}. Furthermore, when
$D$ is the Levi-Civita connection of a K\"ahler metric, then $M$ is
hyperk\"ahler, as in~\cite{Feix}.
\end{prop}
\begin{proof} As noted above, the assumptions of this theorem imply that $\Lp
\to S\hp$ and $\Lm\to S\am$ are trivial. We compute their spaces of affine
sections using the connection $D\in\Pi_c$, so that twisted connections
$D^\nabla$ on $\Lp$ and $\Lm$ are trivial. Furthermore, $D$ has curvature of
type $(1,1)$ if and only if $\Pi_c$ has c-projective curvature of type $(1,1)$
and $\ri^D$ has type $(1,1)$. Thus, in this case, $\ri^D$ vanishes on the
leaves of the $(1,0)$ and $(0,1)$ foliations, and hence a function $f$ on such
a leaf defines an affine section if and only if $D\d f=0$ along the leaf,
i.e., $f$ is an affine function with respect to the flat affine connection
induced by $D$ on the leaf. We conclude that $\Vhp$ and $\Vam$ are vector
bundles dual to the spaces of affine functions along leaves considered by
Feix~\cite{Feix,Feix2}.

It is easy to check that $\php\colon \Sc\times\C\to \Vhp$ and $\phm\colon
\Sc\times\C\to \Vhp$ send $(x,\tx,1)$ to the evaluation maps that Feix uses in
her construction; hence our construction reduces to hers. Because constant
functions are affine, the projection $\hat Z=\Sc\times\C P^1\to \C P^1$
descends to $Z$, which implies $M$ is hypercomplex by~\cite{HKLR}. We refer
to~\cite{Feix} for the proof that $M$ is hyperk\"ahler when $D$ is the
Levi-Civita connection of a K\"ahler metric.
\end{proof}

\subsection{Instantons, twists, and quaternionic complex structures}
\label{s:twist}

In the remaining applications we make use of the quaternionic
generalization~\cite{Joyce2,MaSa,Sal} of the Penrose--Ward correspondence for
self-dual Yang--Mills connections on self-dual conformal
$4$-manifolds~\cite{AHS}. Recall that a $G$-connection $\nabla$ on a
$G$-bundle $V$ over a quaternionic $4n$-manifold $(M,\qs)$ is called a
\emph{$G$-instanton} (or a \emph{quaternionic}, \emph{self-dual} or
\emph{hyperholomorphic} $G$-connection) if its curvature $F^\nabla$ is
$\qs$-hermitian, i.e., $F^\nabla(IX,Y)+F^\nabla(X,IY)=0$ for all $I\in\qs$ and
$X,Y\in TM$. This means equivalently the complexified pullback $\cV$ of $V$ to
$Z$ is holomorphic and of degree zero, i.e., trivial on twistor
lines~\cite{Joyce2,MaSa,Sal}. From the perspective of complexified
quaternionic geometry, if $V\cx\to M\cx$ is the bundle whose fibre over $u\in
M\cx$ is the space of parallel sections over the twistor line $u\sub Z$, then
$\nabla$ extends to a $G\cx$-connection on $V\cx$ which is flat on
$\alpha$-submanifolds, and conversely (taking $M\cx$ to be sufficiently small)
$\cV_z$ is the space of parallel sections of $V\cx$ along $\alpha_z$.

Suppose now that $\tilde G$ acts on $M$ preserving $\qs$ with $\dim\tilde
G=\dim G$, and let $P$ be the principal $G$-bundle with connection
$\omega\colon P\to\g$ induced by $(V,\nabla)$.  Then D.~Joyce
showed~\cite{Joyce2} that for any lift of the $\tilde G$ action to $P$
commuting with $G$, preserving $\omega$, and transverse to $\ker\omega$, the
quotient $P/\tilde G$ is (at least locally) a quaternionic manifold $(\tilde
M,\tilde \qs)$ with a $G$ action preserving $\tilde\qs$.  Joyce gave a
twistorial proof using the induced principal $G\cx$-bundle $\cP\to Z$. Indeed
(omitting technical details), since $\tilde G$ commutes with $G$ and preserves
$\omega$, there is an induced action of $\tilde G\cx$ on $\cP$; now the
transversality condition implies that the image in $\tilde Z:=\cP/\tilde G\cx$
of any section $s$ of $\cP$ over a twistor line $u$ has normal bundle
$\cO(1)\otimes\C^{2n}$, and $\tilde Z$ is then the twistor space of $(\tilde
M,\qs)$.

This method is now known as the \emph{twist construction}, particularly in the
case that $\dim G=1$ or (more generally) $G$ is abelian (see~\cite{MaSw} and
references therein). Here we apply it to generalize some results on self-dual
conformal $4$-manifolds in~\cite{CaPe:sdc}.

To do this, we use the notion, introduced in~\cite{Joyce} and further studied
in~\cite{Pon,Hit3}, of a \emph{quaternionic complex manifold}, which (for us)
is a quaternionic manifold $(M,\qs)$ equipped with a section of $\qs$ defining
an integrable complex structure on $M$. Then $\pm J$ define a divisor
$\cD^{1,0}+\cD^{0,1}$ in the twistor space $Z$ of $M$ and there is a unique
quaternionic connection $D$ with $DJ=0$~\cite{Alex3}. In fact
in~\cite{Joyce,Pon,Hit3}, the authors restrict to the case that $D$ preserves
a volume form, which we prefer to call a \emph{special} quaternionic complex
manifold. As in~\cite{CaPe:sdc}, it is straightforward to see that $(M,\qs,J)$
is special if and only if $[\cD^{1,0}+\cD^{0,1}]=\cO_Z(2)$---where where
$\cO_Z(2n+2)=\Wedge^nTZ$ as in Remark~\ref{rem:EH}---and (\emph{locally})
\emph{hypercomplex} (i.e., $D$ is flat on $\qs$) if and only if
$[\cD^{1,0}-\cD^{0,1}]=\cO$. In general, $\cL_{(s)}:=[\cD^{1,0}+\cD^{0,1}]\tp
\cO_Z(-2)$ and $\cL_{(h)}:=[\cD^{1,0}-\cD^{0,1}]$ are degree zero line bundles
on $Z$, and so correspond to an $\R^+$-instanton $L_{(s)}$ (which is in fact
$D$ on a root of $\Wedge^{4n}TN$) and an $S^1$-instanton $L_{(h)}$ on $M$
(which is in fact $D$ on $J^\perp\sub\qs$).

If $\tilde G$ preserves $J$ as well as $\qs$ in the twist construction, then
$\tilde G\cx$ preserves the inverse image of $\cD^{1,0}+\cD^{0,1}$ in $\cP$,
hence $\tilde M$ is also a quaternionic complex manifold.  Furthermore, if
$(M,\qs,J)$ is special or hypercomplex, and $\tilde G$ preserves the
$D$-parallel sections of $L_{(s)}$ or $L_{(h)}$ respectively, then $\tilde M$
will also be special or hypercomplex accordingly.

When $\dim\tilde G=1$, the $\tilde G$ action always lifts (at least locally on
$M$) but the lift is not unique. In more invariant terms, $\cP\to Z$ has an
action of a complex $2$-torus $\mathbb T\cx$, and its principal bundle
structure over $Z$ is a $\C\mult$ subgroup of $\mathbb T\cx$. Thus there is a
family of twists of $M$ whose twistor spaces are quotients of $\cP$ by other
$\C\mult$ subgroups of $\mathbb T\cx$.

In this case, we can (in particular) take the $G$-bundle over $M$ to be
$L_{(s)}\mult$ or $L_{(h)}\mult$, so that $\cP\to Z$ is either
$\cL_{(s)}\mult$ or $\cL_{(h)}\mult$. Then the pullback $\cR$ of $\cL_{(s)}$
or $\cL_{(h)}$ to $\cP$ has a tautological nonvanishing section and so there
is a homomorphism $\beta\colon\mathbb T\cx\to\C\mult$ via the action on
$H^0(\cP,\cR)\cong\C$.  When this action is trivial all twists are special or
hypercomplex (respectively) and we already considered this situation (for more
general twists) above. Otherwise, the identity component of $\ker\beta$ is a
distinguished $\C\mult$ subgroup of $\mathbb T\cx$ such that (wherever it is
transverse) the quotient $\tilde Z$ is the twistor space of a special
quaternionic complex or hypercomplex manifold respectively. We summarize as
follows.

\begin{prop}\label{p:twist} Let $(M,\qs,J)$ be a quaternionic complex
manifold which is either not special or not hypercomplex, but admits a local
$S^1$ action preserving $\qs$ and $J$. Then there is locally a twist of $M$
\textup(by $L_{(s)}\mult$ or $L_{(h)}\mult$\textup) which is special or
hypercomplex \textup(respectively\textup).
\end{prop}

A special case of this result arises in one direction of the Haydys--Hitchin
correspondence~\cite{Hay,Hit2} between quaternionic K\"ahler and hyperk\"ahler
manifolds with $S^1$ actions. Suppose that $(M,\qs,g)$ is a quaternionic
K\"ahler manifold (of nonzero scalar curvature). Then its twistor space $Z$ is
a holomorphic contact manifold, where the contact distribution is the kernel
of an $\cO_Z(2)$-valued $1$-form $\eta$, invariant under the real structure
$\tau$, and such that (quaternionic) Killing vector fields on $M$ correspond
to $\tau$-invariant sections of $\cO_Z(2)$ by contracting the induced contact
vector field on $Z$ with $\eta$~\cite{Sal0}. Now if $(M,\qs,g)$ has $S^1$
symmetry, the zero set of the section of $\cO_Z(2)$ corresponding to the
generator of the action is a $\C\mult$-invariant degree two divisor which may
be written as $\cD^{1,0}+\cD^{0,1}$. By construction, $\cL_{(s)}$ is trivial,
and so $(M,\qs)$ has a special quaternionic complex structure.  There is
therefore locally a twist of $M$ by $L_{(h)}\mult$ which is hyperk\"ahler with
an $S^1$ action.

\subsection{Twisted Swann bundles and Armstrong cones}

Recall that if $(M,\qs)$ is a quaternionic manifold then the total space of
the principal $\CO(3)$-bundle $\pi_\swb\colon \swb_M\to M$ of oriented
conformal frames $\lambda (J_1,J_2,J_3):\lambda\in\cO_M(1)^+,\,J_i^2 = -\id
=J_1J_2J_3$ of $\cO_M(1)\tp\qs$ has a canonical hypercomplex structure (where
$\cO_M(1)$ is an oriented real line bundle with
$\cO_M(2n+2)=\Wedge^{4n}TM$). Indeed, $\pi_\swb^*TM$ has a tautological
hypercomplex structure, and this lifts to a hypercomplex structure on
$T\swb_M$ using any quaternionic connection $D$ and the hypercomplex structure
on the vertical bundle of $\swb_M$ coming from the isomorphism of $\CO(3)$
with $\HQ\mult/\{\pm1\}$. This construction was introduced by
A.~Swann~\cite{Swann} for quaternionic K\"ahler manifolds, and $\swb_M$ is
called the \emph{Swann bundle} or \emph{hypercomplex cone} of $M$.  The
general case is studied e.g.~in~\cite{Joyce,Joyce2,Sal,PPS}.

In~\cite{Arm}, S.~Armstrong considers cone constructions for other parabolic
geometries of projective type and in particular shows that for any
c-projective manifold $S$, the total space $\cC_S$ of $\cO_S(1)\mult$ carries
a canonical complex affine connection: indeed, as explained in~\cite{CEMN},
$T\cC_S$ is canonically isomorphic to the pullback of the standard
representation of the Cartan connection on $S$ (the \emph{standard tractor
  bundle}). Now our construction shows that any c-projective manifold (with
type $(1,1)$ curvature) is a maximal totally complex submanifold of a
quaternionic manifold $M$ with respect to a section $J$ of $\qs|_S$. The set
of $\lambda(J_1,J_2,J_3)\in\swb|_S$ with $J_3=J$ is a principal
$\C\mult$-subbundle, and this submanifold realises the Armstrong cone $\cC_S$
of $S$ as a maximal totally complex submanifold of $\swb$ with respect to the
third tautological complex structure.

Thus it is natural to expect that the quaternionic Feix--Kaledin and cone
constructions commute, i.e., when $M$ is obtained by applying the quaternionic
Feix--Kaledin construction to $S$, its hypercomplex cone is given locally by
applying Feix's hypercomplex construction to the Armstrong cone of $S$.  To
see this, we will in fact show more: that the same holds for twisted versions
of both constructions.

In~\cite{Joyce2,PPS}, it was observed that the Swann bundle construction can
be twisted by an $\R^+$-instanton (an oriented real hyperholomorphic line
bundle $L$): one can replace $\cO_M(1)\tp\qs$ with $L\tp\cO_M(1)\tp\qs$ above
to obtain a hypercomplex manifold $\swb_L$ called a \emph{twisted Swann
  bundle}. On the other hand, if $\cL\to S$ is a complex line bundle with
connection $\nabla$, the \emph{twisted Armstrong cone} is $\cC_\cL=
(\cL\otimes \cO_S(1))\mult$, with the affine connection induced by the tensor
product of the standard tractor connection and $\nabla$.

\begin{thm} \label{Swann}
Let $(M,\qs)$ be obtained from the quaternionic Feix--Kaledin construction
applied to the c-projective manifold $(S,\Pi_c)$ of type $(1,1)$ and the line
bundle $\cL$ with connection $\nabla$ of type $(1,1)$ as in
Theorem~\textup{\ref{maintheorem}}. Then the hypercomplex manifold obtained
from twisted Armstrong cone $\cC_\cL$ by Feix's hypercomplex construction is an
open subset of the twisted Swann bundle of $M$, where the pullback of
$\cL_Z^{\;2}$ to $\hat Z$ is $p^*(\cL\tp\overline\cL)$.
\end{thm}
\begin{proof} We prove this result by identifying the twistor spaces.
As observed by Hitchin~\cite{Hit2} in the untwisted quaternionic K\"ahler
case, if $Z$ is the twistor space of $M$, $\cL_Z$ is the Penrose--Ward
transform of $L$ and $\cO_Z(2n+2)=\Wedge^nTZ$ as in Remark~\ref{rem:EH} then
the twistor space of $\swb_L$ is $(\C^2\tp\cL_Z\tp\cO_Z(1))\mult/\{\pm1\}$:
this space has a natural $\C\mult$ action induced by scalar multiplication on
$\C^2$, and the quotient is $Z\times \C\Proj^1$.

On the other hand, the principal $\C\mult\times\C\mult$ bundle $\cC_\cL\cx:=
\Lp\mult\times\Lm\mult\to \Sc$ is a complexification of the twisted Armstrong
cone $\cC_\cL=(\cL\otimes \cO_S(1))\mult$, and so $\Lp^*\oplus\Lm^*\to\Sc$ is
an associated bundle $\cC_\cL\cx\times_{\C\mult\times\C\mult}\C^2$ with
projectivization $\hat Z=\cC_\cL\cx\times_{\C\mult\times\C\mult} \C\Proj^1$
(where the diagonal subgroup acts trivially on $\C\Proj^1$). Thus
$\cC_\cL\cx\times\C\Proj^1$ is a principal $\C\mult\times\C\mult$ bundle over
$\hat Z$. Hence the Feix twistor space of $\cC_\cL$ is an open subset of a
twist of $Z\times\C^2$ (where $Z$ is the twistor space of $M$) by a line
bundle of degree one (so that the twistor lines in $Z$ lift to twistor
lines). This line bundle therefore has the form $\cL_Z\tp\cO_Z(1)$ for some
degree zero line bundle $\cL_Z$, and the pullback of $\cL_Z\tp\cO_Z(1)$ by
$\phi\colon \hat Z\to Z$ must be $\cO_{\Lp^*\ds\Lm^*}(1)$.

We now take top exterior powers of the short exact sequence
\[
0\rightarrow Vp\rightarrow\phi^*TZ \rightarrow \Nb\hp\ds\Nb\am\rightarrow 0.
\]
over $\hat Z$ to obtain
\[
\phi^*\cO_Z(2n+2)=\phi^*(\Wedge^{2n+1}TZ)
\cong Vp\tp\Wedge^n\Nb\hp\tp\Wedge^n\Nb\am,
\]
where the vertical bundle $Vp$ to the fibres of $p\colon \hat Z\to \Sc$
satisfies
\[
Vp\cong\cO_{\Lp^*\ds\Lm^*}(2)\otimes p^*\Lp^*\otimes p^*\Lm^*,
\]
and therefore (using Proposition~\ref{p:normal})
\begin{align*}
\phi^*\cO_Z(2n+2)&\cong
\cO_{\Lp^*\ds\Lm^*}(2n+2)\tp p^*(\Wedge^nTS^{1,0}\tp\Lp^{-(n+1)}
\tp\Wedge^nTS^{0,1}\tp\Lm^{-(n+1)})\\
&\cong\cO_{\Lp^*\ds\Lm^*}(2n+2)\tp p^*(\cL\tp\overline\cL)^{-(n+1)}.
\end{align*}
We conclude that the Feix twistor space of $\cC_\cL$ is a double
cover of an open subset of the twisted Swann bundle twistor space, with
$\phi^*\cL_Z^{\;2}\cong p^*(\cL\tp\overline\cL)$ as required.
\end{proof}

\subsection{Quaternionic K\"ahler metrics and the Haydys--Hitchin
correspondence}

The quaternionic Feix--Kaledin construction produces a hyperk\"ahler metric or
hypercomplex structure on $M$ when it reduces to the original construction by
Feix. It is natural to ask when the quaternionic manifold $M$ admits an
$S^1$-invariant quaternionic K\"ahler metric of nonzero scalar curvature. For
example, we have seen that the quaternionic K\"ahler symmetric spaces
$\HQ\Proj^n$ and $\Gr_2(\C^{2n+2})$ may be constructed locally from the flat
c-projective structure on $\C\Proj^n$ using different twists.

Quaternionic K\"ahler manifolds $(M,\qs,g)$ with $S^1$ actions have been
studied by A.~Haydys and N.~Hitchin~\cite{Hay,Hit2} who associate to any such
manifold a hyperk\"ahler manifold with a non-triholomorphic $S^1$ action.  As
special cases, the Haydys--Hitchin correspondence relates the rigid c-map
construction of semi-flat hyperk\"ahler metrics to the quaternionic K\"ahler
c-map~\cite{ACDM,CFG,Hit2,MaSw}, and it generalizes the link between
$S^1$-invariant self-dual Einstein manifolds of nonzero and zero scalar
curvature~\cite{FiSa,Prz,Tod}.

The quaternionic Feix--Kaledin construction complements these methods (which
generally apply on the open subset where the $S^1$ action is locally free) by
describing the correspondence on a neighbourhood of a maximal totally complex
submanifold of fixed points of the $S^1$ action.

\begin{thm}\label{thm:qK} Let $(S,J,g)$ be a K\"ahler--Einstein $2n$-manifold
  of nonzero scalar curvature
  with the c-projective structure and connection on $\cO_S(1)$ induced by the
  Levi-Civita connection.  Then the quaternionic Feix--Kaledin construction,
  with $\cL=\cO_S(k)$ a tensor power of $\cO_S(1)$, yields
  \textup(locally\textup) a quaternionic K\"ahler manifold $M_k$ in the
  Haydys--Hitchin family associated to the hyperk\"ahler manifold $M_{-1}$
  obtained from the Feix--Kaledin construction.
\end{thm}
\begin{proof} Since $g$ is K\"ahler--Einstein, the normal Cartan connection
of the c-projective structure preserves a metric on the standard tractor
bundle (see e.g.~\cite[Prop.~4.8]{CEMN}) and hence so does its twist by a
unitary connection. Since this connection is also torsion-free, the twisted
Armstrong cone $\cC_k=\cL \otimes\cO_S(1)$ of $S$ is K\"ahler, so the
Feix--Kaledin construction yields a hyperk\"ahler manifold, which is an open
subset of the (untwisted) Swann bundle of $M_k$ by Theorem~\ref{Swann}, since
$\cL$ is a unitary bundle.  Furthermore, for $k\neq -1$ each $\cC_k$ is
covered by $\cC_0$, so the Swann bundles are all locally isomorphic to a fixed
hyperk\"ahler manifold $\swb$. The circle actions on $M_k$ lift to a
triholomorphic circle action on $\swb$, which preserves the Obata connection,
i.e., the Levi-Civita connection of the hyperk\"ahler metric. However, a
homothetic circle action must be isometric. It follows that each $M_k$ admits
an $S^1$-invariant quaternionic K\"ahler metric~\cite{Swann}.

To see that $M_{-1}$ is (locally) the hyperk\"ahler manifold $\tilde M$ in the
family, we use the twist construction of the latter from $Z_k$ as
in~\S\ref{s:twist} and~\cite{Hit2}. Thus the twistor space
$\tilde\zeta\colon\tilde Z\to \C\Proj^1$ of $\tilde M$ is a $\C\mult$ quotient
of $\cL_k\mult$ where $\cL_k\to Z_k$ is the divisor line bundle of
$\cD^{1,0}-\cD^{0,1}$ and $\cD^{1,0}+\cD^{0,1}$ is the zero-set of the section
of $\cO_Z(2)$ corresponding to $S^1$ action on $M_k$.  In particular,
$S^{1,0}\sub\cD^{1,0}$, $S^{0,1}\sub\cD^{0,1}$ and $\tau(\cD_0)=\cD_\infty$.
The vertical $\C\mult$ action on $\cL_k\mult\to Z$ descends to a $\C\mult$
action on $\tilde Z$ preserving the divisor $\tilde\cD^{1,0}+\tilde\cD^{0,1}
=\tilde\zeta^{-1}(\{0\}+\{\infty\})$, and fixing copies of $S^{1,0}$ and
$S^{0,1}$. Thus the induced $S^1$-action on $\tilde M$ preserves the complex
structures $\pm J$ in the hyperk\"ahler family corresponding to these
divisors, and fixes a copy of $S$ which is maximal totally complex with
respect to $\pm J$. As in proof of Theorem~\ref{conversetheorem}, it now
follows that the blow-up of $\tilde Z$ along $S^{1,0}\sqcup S^{0,1}$ is
locally isomorphic to $S\cx\times \C\Proj^1$, where the $\C\Proj^1$-bundle in
Theorem~\ref{conversetheorem} has been trivialized by the pullback of
$\tilde\zeta$ to the blow-up of $\tilde Z$. Hence $\tilde M$ is locally
isomorphic to $M_{-1}$.
\end{proof}

We end by remarking that this theorem should generalize to the case that $S$
is merely a c-projective manifold of type $(1,1)$ and $\cL=\cO_S(k)$ is a
tensor power of $\cO_S(1)$, yielding a family quaternionic manifolds $M_k$
with $M_{-1}$ hypercomplex.

\acknowledge We would like to thank Roger Bielawski, Fran Burstall, Nigel
Hitchin, Lionel Mason, Andrew Swann and the anonymous referees for helpful
comments which led to improvements in both the content and the presentation of
the paper.  We also thank the Eduard Cech Institute and the Czech Grant
Agency, grant nr.~P201/12/G028, for hospitality and financial support.

\end{document}